\documentclass[a4paper, 10pt]{amsart}
\usepackage{amsthm}
\usepackage[]{amsmath}
\usepackage{amssymb}
\usepackage{enumerate}
\usepackage{tabularx}
\usepackage[]{color}
\usepackage[left=2.5cm, right=2.5cm]{geometry}
\usepackage[colorlinks]{hyperref}
\usepackage{tikz}
\usepackage{multirow}
\usepackage{subcaption}

\newcommand{\bm}[1]{\boldsymbol{#1}}
\newcommand{\bmr}[1]{\bm{\mr{#1}}}
\newcommand{\lj}{[ \hspace{-2pt} [}
\newcommand{\rj}{] \hspace{-2pt} ]}
\newcommand{\mb}[1]{\mathbb{#1}}
\newcommand{\mc}[1]{\mathcal{#1}}
\newcommand{\mr}[1]{\mathrm{#1}}
\newcommand{\jump}[1]{\lj #1 \rj}
\newcommand{\aver}[1]{ \{#1\}  }
\newcommand{\enorm}[1]{ |\!|\!| #1 |\!|\!|}

\newcommand{\wt}[1]{{#1}^{\bmr{p}}}
\newcommand{\wh}[1]{\widehat{#1}}
\newcommand\MTh{\mc{T}_h}
\newcommand\MEh{\mc{E}_h}
\newcommand\un{\bm{\mr n}}
\newcommand\curl{\ifmmode \mathrm{curl} \else \text{curl}\fi}
\renewcommand{\d}[1]{\mathrm d \boldsymbol{#1}}
\newcommand{\vecd}[2]{\begin{pmatrix} #1 \\ #2 \\ \end{pmatrix}}

\newtheorem{assumption}{Assumption}
\newtheorem{theorem}{Theorem}
\newtheorem{lemma}{Lemma}

\newtheorem{remark}{Remark}

\title[Sequential Least Squares]{A Sequential Least Squares Method for
  Poisson Equation using A Patch Reconstructed Space}

\author[R. Li]{Ruo Li} \address{CAPT, LMAM and School of Mathematical
  Sciences, Peking University, Beijing 100871, P.R. China}
\email{rli@math.pku.edu.cn}

\author[F.-Y. Yang]{Fanyi Yang} \address{School of Mathematical
  Sciences, Peking University, Beijing 100871, P.R. China}
\email{yangfanyi@pku.edu.cn}

\begin{document}

\maketitle


\begin{abstract}
  We propose a new least squares finite element method to solve the
  Poisson equation. By using a piecewisely irrotational space to
  approximate the flux, we split the classical method into two
  sequential steps. The first step gives the approximation of flux in
  the new approximation space and the second step can use flexible
  approaches to give the pressure. The new approximation space for
  flux is constructed by patch reconstruction with one unknown per
  element consisting of piecewisely irrotational polynomials. The
  error estimates in the energy norm and $L^2$ norm are derived for
  the flux and the pressure. Numerical results verify the convergence
  order in error estimates, and demonstrate the flexibility and
  particularly the great efficiency of our method.

  \noindent \textbf{keywords}: Poisson equation, Patch
  reconstructed, Irrotational polynomial space, Discontinuous
  least squares finite element method.

\end{abstract}


\section{Introduction}
The least squares finite element method (LSFEM) is a sophisticated
technique for solving the partial differential equation. For
second-order elliptic problems, we refer to \cite{Bochev2009least,
  Jiang1993optimal, Bramble1997least, Pehlivanov1994least,
  Aziz1985least}, for the Navier-Stokes problem, we refer to
\cite{Bochev1993accuracy, Bochev2012locally, Lung1994stokes}.  For an
overview of the least squares finite element methods, we refer to
\cite{Bochev1998review} and the references therein. Different from the
Galerkin method, the lease squares method is based on the minimization
of the $L^2$-norm residual over a proper approximation space. An
immediate advantage is the symmetric positive definite resulting
linear system, which has made the method attractive in several
fields. Instantly, one may see the condition number of the resulting
linear system is squared due to the formation of the approximation. To
relieve the curse due to the condition number, one may write the
equation into low order formation. Taking the Poisson equation as an
example, we may introduce a flux variable to write it into the mixed
formation, resulting a system coupled by the flux and pressure. Though
the mixed form is helpful in reducing condition number, more degree of
freedoms(DOF) are introduced to achieve the same accuracy.

Discontinuous Galerkin(DG) methods have received massive attention in
the past two decades due to its great flexibility in mesh partition
and easy implementation of the approximation spaces especially for the
spaces of high order. We refer to the review paper
\cite{arnold2002unified} and the references therein. Using the
approximation space from the DG methods, discontinuous least squares
(DLS) finite element methods have been developed in
\cite{Ye2018discontinuous, Bensow2005discontinuous, Bensow2005div} for
solving the elliptic system. In \cite{Bochev2012locally,
  Bochev2013nonconforming}, the authors extend the DLS finite element
methods to the Stokes problem in velocity-vorticity-pressure form. The
same as the least squares methods using continuous approximation
space, the technique to write the equation into low order system is
adopted in DLS methods either to reduce the condition number of
resulting linear systems. To achieve the high order accuracy,
discontinuous finite element space requires a huge number of degrees
of freedom which leads to a very large linear system
\cite{hughes2000comparison, zienkiewicz2003discontinuous} in
comparison to the methods using continuous approximation spaces. The
coupling of the variables in the mixed form and the increasing of the
number of DOFs make one hard to satisfy with its efficiency.

In this paper, a new least squares finite element method is proposed
to solve the Poisson equation. The novel point is that we split the
solver into two sequential steps. This is motivated from the idea in
\cite{Bensow2005div} to decouple the least-squares-type functional
into two subproblems. In the first step, we approximate the flux still
using a discontinuous approximation space. This space is the piecewise
irrotational polynomial space, which is a generalization of the
reconstructed space proposed in \cite{li2012efficient,
  li2016discontinuous}. The new space is obtained by solving a local
least squares problem based on the irrotational polynomial bases and
only one unknown locates inside each element. With such a space, we
makes the idea in \cite{Bensow2005div} to decouple the flux and the
pressure implementable. For the flux, the optimal error estimate with
respect to the energy norm is derived. We can only prove the
suboptimal convergence rate in $L^2$ norm for the flux until now,
while in numerical experiments we obverse the optimal convergence
behavior for the space of odd order.

Once we get the numerical approximation to the flux, one then could
use the numerical flux to obtain the pressure in a very flexible
manner. As a demonstration, we adopt the standard $C^0$ finite element
space to solve the pressure. We give the error estimates of the
pressure in both energy norm and $L^2$ norm. By a series of numerical
examples, we at first verify the convergence order given in the error
estimate and illustrate the flexibility we inherit from the DG
method. Particularly, by the comparison \cite{hughes2000comparison} of
the number of DOFs used to achieve the same
numerical error, we show that our method has a great saving in DOFs
compared to the standard DLS finite element method. Consequently, by
the decoupling of the flux and the pressure and by the saving in the
number of DOFs, a much better efficiency could be attained by our
method.

The rest of this paper is organized as follows. In Section
\ref{sec:fems}, we review the standard DLS finite element method and
present the corresponding error estimates. In Section
\ref{sec:irrotationbasis}, we introduce a reconstruction operator to
define the piecewise irrotational approximation space and we give the
approximation property of the new space. In Section \ref{sec:lsfem},
the approximation to the flux and the pressure of the Poisson problem
is proposed, and we derive the error estimates for both flux and
pressure in energy norm and $L^2$ norm. In Section
\ref{sec:numericalresults}, we present the numerical examples on
meshes with different geometry to verify the convergence order in the
error estimates. Besides, we make a comparison of number of DOFs
respect to the numerical error between our method and the method in
Section \ref{sec:fems} to show the great efficiency of our method.


\section{Discontinuous Least Squares Finite Element Method}
\label{sec:fems}
Let $\Omega$ be a bounded polygonal domain in $\mathbb{R}^d(d=2, 3)$.
Let $\MTh$ be a partition of $\Omega$ into polygonal (polyhedral)
elements. We denote by $\MEh^i$ the set of interior element faces of
$\MTh$ and by $\MEh^b$ the set of the element faces on the boundary
$\partial \Omega$, thus the set of all element faces $\MEh = \MEh^b
\cup \MEh^i$. The diameter of an element $K$ is denoted by $h_K =
\text{diam}(K)$, $\forall K \in \MTh$ and the size of the face $e$ is
$h_e = |e|$, $\forall e \in \MEh$. We denote $h = h_{\max} = \max_{K
\in \MTh} h_K$. It is assumed that the elements in $\MTh$ are
shape-regular according to the conditions specified in
\cite{antonietti:2013}, which read: {\it there are
\begin{itemize}
  \item two positive numbers $N$ and $\sigma$ which are independent of
    $h$;
  \item a compatible sub-decomposition $\widetilde{\mc T}_h$
    consisting of shape-regular triangles;
\end{itemize}
such that
\begin{itemize}
\item any element $K \in \MTh$ admits a decomposition $\widetilde{\mc
  T}_{h|K}$ which is composed of less than $N$ shape-regular
  triangles;
\item the triangle $\widetilde{K} \in \widetilde{\mc T}_h$ is
  shape-regular in the sense of that the ratio between
  $h_{\widetilde{K}}$ and $\rho_{\widetilde{K}}$ is bounded by
  $\sigma$: $h_{\widetilde{K}} / \rho_{\widetilde{K}} \leq \sigma $
  where $\rho_{\widetilde{K}}$ is the radius of the largest ball
  inscribed in $\widetilde{K}$.
\end{itemize}}
The regularity conditions could lead to some useful consequences
which are easily verified: 
\newcounter{regularity}
\setcounter{regularity}{1}
\begin{enumerate}
\item[M\arabic{regularity}] There exists a positive constant
  $\sigma_s$ such that $\sigma_v h_K \leq h_e$ for any element $K$
  and every edge $e$ of $K$;
  \addtocounter{regularity}{1}
\item[M\arabic{regularity}][{\it trace inequality}] There exists a
  positive constant $C$ such that 
  \begin{equation}
    \|v\|_{L^2(\partial K)}^2 \leq C \left( h_K^{-1} \| v\|_{L^2(
        K)}^2 + h_K \| \nabla v\|_{L^2(K)}^2 \right), \quad \forall v
    \in H^1(K).
    \label{eq:traceinequality}
  \end{equation}
  \addtocounter{regularity}{1}
\item[M\arabic{regularity}][{\it inverse inequality}] There exists a
  positive constant $C$ such that
  \begin{equation}
    \| \nabla v\|_{L^2(K)} \leq Ch_K^{-1} \|v\|_{L^2(K)}, \quad
    \forall v \in \mb P_m(K),
    \label{eq:inverse}
  \end{equation}
  where $\mb P_m(\cdot)$ is the polynomial space of degree $\leq m$.
\end{enumerate}
Next, we introduce the standard trace operators in the discontinuous
Galerkin (DG) framework \cite{arnold2002unified}. Let $v$ be a scalar-
or vector-valued function and $e \in \MEh^i$ shared by two adjacent
elements $K^+$ and $K^-$ with the unit outward normal $\un^+$ and
$\un^-$ corresponding to $\partial K^+$ and $\partial K^-$,
respectively. We define the average operator $\aver{\cdot}$ and the
jump operator $\jump{\cdot}$ as
\begin{displaymath}
  \aver{v} = \frac{1}{2}\left( v|_{K^+} + v|_{K^-} \right), \quad
  \forall e \in \MEh^i,
\end{displaymath}
and
\begin{displaymath}
  \jump{v} = v|_{K^+}\un^+ + v|_{K^-}\un^-, \quad \jump{v \otimes \un}
  = v|_{K^+}\otimes \un^+ + v|_{K^-} \otimes \un^-, \quad \forall e
  \in \MEh^i.
\end{displaymath}
In the case $e \in \MEh^b$, $\aver{\cdot}$ and $\jump{\cdot}$ are
modified as 
\begin{displaymath}
  \aver{v} = v, \quad \jump{v} = v\un, \quad \jump{v \otimes \un} = v
  \otimes \un, \quad \forall e \in \MEh^b,
\end{displaymath}
where $\un$ denotes the unit outward normal to $e$. 

Throughout the paper, let us note that $C$ and $C$ with a subscript
are generic constants which may be different from line to line but are
independent of the mesh size, and we follow the standard definitions
for the spaces: $L^2(D)$, $H^t(D)$, $C^t(D)$, $\bmr{L}^2(D) :=
[L^2(D)]^d$, $\bmr{H}^t(D) = [H^t(D)]^d$, $\bmr{C}^t(D) = [C^t(D)]^d(t
\geq 0)$ and we define 
\begin{displaymath}
  H(\text{curl}^0; D) \triangleq \left\{ \bm v \in \bmr{L}^2(D)\ |\
  \nabla \times \bm v = 0 \right\}.
\end{displaymath}

The problem considered in this article is the Poisson's equation: {\it
seek $u$ such that}
\begin{equation}
  \begin{aligned}
    - \Delta u &= f, \quad \text{in}\ \Omega, \\
    u &= g, \quad \text{on}\ \partial \Omega. \\
  \end{aligned}
  \label{eq:elliptic}
\end{equation}
The first step of usual least squares finite element methods
\cite{Ye2018discontinuous, Bensow2005discontinuous} is to write the
problem \eqref{eq:elliptic} into an equivalent mixed form: {\it seek
$\bmr p$ and $u$ such that
\begin{equation}
  \begin{aligned}
    \bmr p -  \nabla u &= \bm 0, \quad \text{in}\ \Omega, \\
    -\nabla \cdot \bmr p &= f, \quad \text{in}\ \Omega, \\
    u &= g, \quad \text{on}\  \partial \Omega.
  \end{aligned}
  \label{eq:ellipticmixed}
\end{equation}}
In the mixed form, we refer $u$ as the pressure and $\bmr p$ as the
flux later on based on the terminology of the background of this
equation in fluid dynamics. Here we introduce two discontinuous
approximation spaces: $V_h^m$ for the pressure $u$ and $\bmr{W}_h^m$
for the flux $\bmr q$, which are defined as below:
\begin{displaymath}
  \begin{aligned}
    V_h^m &= \left\{ v_h \in L^2(\Omega)\ |\ v_h|_K \in \mb P_m(K), \
    \forall K \in \MTh\right\}, \\
    \bmr{W}_h^m &= \left\{ \bmr{q}_h \in \bmr{L}^2(\Omega) \ |\
    \bmr{q}_h |_K \in \left[ \mb P_m(K) \right]^d, \ \forall K \in
    \MTh \right\}, \\
  \end{aligned}
\end{displaymath}
where $m$ is a positive integer. We equip these two approximation
spaces with the following norms, $\|\cdot\|_{u}$ for $V_h^m$ and
$\|\cdot\|_{\bmr{p}}$ for $\bmr{W}_h^m$, respectively, as
\begin{displaymath}
  \begin{aligned}
    \|v_h\|_{u}^2 \triangleq & \sum_{K \in \MTh} \|\nabla
    v_h\|_{L^2(K)}^2 + \sum_{e \in \MEh} h^{-1} \|
    \jump{v_h}\|_{L^2(e)}^2, \quad \forall v_h \in V_h^m,\\
    \|\bmr{q}_h\|_{\bmr{p}}^2 \triangleq &\sum_{K \in \MTh} \left( \|
    \nabla \cdot \bmr{q}_h\|_{L^2(K)}^2 + \|\bmr{q}_h\|_{L^2(K)}^2
    \right) + \sum_{e \in \MEh^i} h^{-1} \| \jump{\bmr{q_h}}
    \|_{L^2(e)}^2, \quad \forall \bmr{q}_h \in \bmr{W}_h^m. \\
  \end{aligned}
\end{displaymath}

The standard least squares finite element method based on mixed form
\eqref{eq:ellipticmixed} reads \cite{Ye2018discontinuous}: {\it find
  $(u_h, \bmr{p}_h) \in V_h^m \times \bmr{W}_h^m$ such that
\begin{equation}
  J_h(u_h, \bmr{p}_h) = \inf_{ (v_h, \bmr{q}_h) \in V_h^m \times
  \bmr{W}_h^m} J_h(v_h, \bmr{q}_h),
  \label{eq:infdg}
\end{equation}
where $J_h(\cdot, \cdot)$ is the least squares functional which is
defined as 
\begin{equation}
  \begin{aligned}
    J_h(v_h, \bmr{q}_h)& \triangleq \sum_{K \in \MTh} \left( 
    \| \nabla \cdot \bmr{q}_h + f\|_{L^2(K)}^2 + \| \nabla v_h -
    \bmr{q}_h \|_{L^2(K)}^2\right) \\
    & + \sum_{e \in\MEh^i} \frac{1}{h} \| \jump{v_h}\|_{L^2(e)}^2
    + \sum_{e \in \MEh^i} \frac{1}{h} \| \jump{\bmr{q_h} \otimes
    \un} \|_{L^2(e)}^2 + \sum_{e \in \MEh^b} \frac{1}{h} \| v_h - g
    \|_{L^2(e)}^2. \\
  \end{aligned}
  \label{eq:infdgep}
\end{equation}}
To solve the minimization problem \eqref{eq:infdg}, one has its
corresponding variational equation which takes the form: {\it find $( u_h,
\bmr{p}_h) \in V_h^m \times \bmr{W}_h^m$ such that
\begin{equation}
  a_h(u_h, \bmr{p}_h; v_h, \bmr{q}_h) = l_h(v_h, \bmr{q}_h), \quad
  \forall (v_h, \bmr{q}_h) \in V_h^m \times \bmr{W}_h^m, 
  \label{eq:dgweakform}
\end{equation}
where the bilinear form $a_h(\cdot; \cdot)$ and the linear form
$l_h(\cdot)$ are defined by
\begin{displaymath}
  \begin{aligned}
    a_h(u_h, &\bmr{p}_h; v_h, \bmr{q}_h) = \sum_{K \in \MTh} \left(
    \int_K \nabla \cdot \bmr{p}_h \nabla \cdot \bmr{q}_h \d{x} +
    \int_K (\nabla u_h - \bmr{p}_h)(\nabla v_h - \bmr{q}_h) \d{x}
    \right) \\
    &+ \sum_{e \in\MEh^i} \int_e \frac{1}{h} \jump{u} \jump{v} \d{s}
    +  \sum_{e \in \MEh^i} \int_e \frac{1}{h} \jump{\bmr{p}_h \otimes
    \un} \jump{ \bmr{q}_h \otimes \un} \d{s} + \sum_{e \in \MEh^b}
    \int_e \frac{1}{h} u_h v_h \d{s}, \\
  \end{aligned}
\end{displaymath}
and
\begin{displaymath}
  l_h(v_h, \bmr{q}_h) = \sum_{K \in \MTh} \int_K f \nabla \cdot
  \bmr{q}_h \d{x} + \sum_{e \in \MEh^b} \frac{1}{h} \int_e g v_h
  \d{s}.
\end{displaymath}}
The coercivity of the bilinear form $a_h(\cdot; \cdot)$ are given in
\cite[Lemma 3.1]{Ye2018discontinuous} as
\begin{lemma}
  For any $(v_h, \bmr{q}_h) \in V_h^m \times \bmr{W}_h^m$, there
  exists a constant $C$ such that 
  \begin{equation}
    a_h(v_h, \bmr{q}_h; v_h, \bmr{q}_h) \geq C\left( \|v_h\|_{u}^2 +
    \|\bmr{q}_h\|_{\bmr{p}}^2 \right).
    \label{eq:dgcoercivity}
  \end{equation}
  \label{le:dgcoercivity}
\end{lemma}
The uniqueness of the solution to \eqref{eq:dgweakform} instantly
follows from Lemma \ref{le:dgcoercivity} and the trivial boundedness
of $a_h(\cdot; \cdot)$. Further, it is direct to derive the error
estimate with respect to the norms $\| \cdot\|_{u}$ and
$\| \cdot \|_{\bmr{p}}$ by the approximation properties of spaces
$V_h^m$ and $\bmr{W}_h^m$ \cite[Theorem 4.1]{Ye2018discontinuous}.
\begin{theorem} 
  Let $u_h \times \bmr{q}_h \in V_h^m \times \bmr{W}_h^m$ be the
  solution to \eqref{eq:dgcoercivity}, and assume the exact solution
  $u \in H^{m+1}(\Omega)$ and $\bmr{q} \in \bmr{H}^{m+1}(\Omega)$,
  then there exists a constant $C$ such that
  \begin{equation}
    \|u - u_h\|_{u} + \|\bmr{q} - \bmr{q}_h\|_{\bmr{p}} \leq Ch^{m} \left(
    \|u\|_{H^{m+1}(\Omega)} + \|\bmr{q} \|_{\bmr{H}^{m+1}(\Omega)}
    \right).
    \label{eq:dgestimate}
  \end{equation}
  \label{th:dgestimate}
\end{theorem}


\section{Approximation Space with Irrotational Basis}
\label{sec:irrotationbasis}
In this section, we follow the idea in \cite{li2016discontinuous,
  li2017discontinuous} to define an approximation space using a patch
reconstruction operator. Purposely, the reconstruction operator we
propose here will use the irrotational basis, thus the approximation space
obtained is piecewise rotation free. With this new approximation
space, we will decouple the minimization problem \eqref{eq:infdgep}
into two sub-problems, that we can numerically solve $\bmr{p}$ at
first and then solve $u$. Let us introduce an irrotational space
$\bmr{S}_m$ which plays a key role in the construction of the
operator,
\begin{displaymath}
  \bmr{S}_m(D) = \left\{ \bmr{v} \in [\mathbb P_m(D)]^d\ |\ 
  \nabla \times \bmr{v} = 0\right\}.
\end{displaymath}
For the irrotational space, we have that:
\begin{lemma}
  For $\forall \bmr{q} \in \bmr{H}^{m+1}(K) \cap H(\curl^0, K)$,
  there exists a constant $C$ such that there is a polynomial
  $\widetilde{\bmr{q}}_h \in \bmr{S}_m(K)$ such that 
  \begin{equation}
    \| \bmr{q} - \widetilde{\bmr{q}}_h\|_{L^2(K)} + h_K \| \nabla
    \left( \bmr{q} - \widetilde{\bmr{q}}_h \right) \|_{L^2(K)} \leq C
    h_K^{m+1} \|\bmr{q}\|_{ \bmr{H}^{m+1}(K)}.  
    \label{eq:appirr}
  \end{equation}
  \label{le:appirr}
\end{lemma}
\begin{proof}
  Since $H(\curl^0, K) = \nabla H^1(K)$ \cite{girault1986finite},
  there exists a $v \in H^{m+2}(K)$ such that $\bmr{q} = \nabla
  v$. Let $\widetilde{v} \in \mb P_{m+2}(K)$ be the standard nodal
  interpolation polynomial of $v$, and let
  $\widetilde{\bmr{q}}_h = \nabla \widetilde{v}_h$.  The inequality
  \eqref{eq:appirr} directly follows from the approximation properties
  of $\widetilde{v}_h$.
\end{proof}

With the partition $\MTh$, we define a reconstruction operator from
$\bmr{C}^0(\Omega)$ to the piecewise irrotational polynomial
space. For any element $K \in \MTh$, we prescribe a point
$\bm x_K \in K$, referred as the {\it sampling node} later on, which
is preferred to be the barycenter of $K$. Then, for each element $K$
we construct an element patch $S(K)$ which is an agglomeration of
elements that contain $K$ itself and some elements around $K$. There
are a variety of approaches to build the element patch and in this
paper we agglomerate elements to form the element patch
recursively. For element $K$, we first let
$S_0(K) = \left\{ K \right\}$ and we define $S_t(K)$ as
\begin{displaymath}
  S_t(K) = S_{t-1}(K) \cup \left\{ K'\ |\ \exists \widetilde{K} \in
  S_{t-1}(K)\ \text{s.t.}\ K' \cap \widetilde{K} = e \in \MEh\right\},
  \quad t = 1, 2, \cdots
\end{displaymath}
In the implementation of our code, at the depth $t$ we enlarge $S_t(K)$
element by element and once $S_t(K)$ has collected sufficiently large
number of elements we stop the recursive procedure and let $S(K) =
S_t(K)$, otherwise we let $t = t + 1$ and continue the recursion. The
cardinality of $S(K)$ is denoted by $\# S(K)$.

Further, for element $K$ we denote by $\mc I_K$ the set of sampling
nodes located inside the element patch $S(K)$,
\begin{displaymath}
  \mc I_K \triangleq \left\{ \bm x_{\widetilde{K}}\ |\ \forall
  \widetilde{K} \in S(K) \right\}.
\end{displaymath}
For any function $\bmr{f} \in \bmr{C}^0(\Omega) \cap H(\curl^0;
\Omega)$ and an element $K \in \MTh$, we seek a polynomial $\mc{R}^m_K
\bmr{f}$ of degree $m$ defined on $S(K)$ by solving the following
least squares problem:
\begin{equation}
  \mc{R}^m \bmr{f} = \mathop{\arg \min}_{\bmr{v} \in \bmr{S}_m(S(K))}
  \sum_{\bm x_{\widetilde{K}} \in \mc I_K } |\bmr{v}(\bm
  x_{\widetilde{K}}) - \bmr{f}(\bm x_{\widetilde{K}})|^2.
  \label{eq:lsproblem}
\end{equation}
We note that the existence of the solution to \eqref{eq:lsproblem} is
obvious but the uniqueness of the solution depends on the position of
the sampling nodes in $\mc I_K$, here we follow
\cite{li2016discontinuous} to state the following assumption:
\begin{assumption}
  For all element $K \in \MTh$ and $\bmr{v} \in \bmr{S}_m(S(K))$, 
  \begin{displaymath}
    \bmr{v}|_{\mc I_K} = \bmr{0} \quad \text{implies} \quad
    \bmr{v}|_{S(K)} \equiv \bmr{0}.
  \end{displaymath}
\end{assumption}
This assumption demands the number $\# S(K)$ shall be greater than
$\text{dim}(\bmr{S}_m) / d$ and excludes the situation that all the
points in $\mc I_K$ lie on an algebraic curve of degree $m$.
Hereafter, we always require the assumption holds. 

Due to the linear dependence of the solution \eqref{eq:lsproblem}, a
global reconstruction operator $\mc{R}^m$ for $\bmr{f}$ can be defined
by restricting the polynomial $\mc{R}^m_K \bmr{f}$ on $K$:
\begin{displaymath}
  (\mc{R}^m \bmr{f})|_K = (\mc{R}^m_K \bmr{f})|_K, \quad \forall K \in \MTh.
\end{displaymath}
It is clear that the operator $\mc{R}^m$ embeds the space
$\bmr{C}^0(\Omega) \cap H(\curl^0; \Omega)$ to a piecewise
irrotational polynomial space of degree $m$, and we denote by
$\bmr{U}_h^m$ the image of the operator $\mc{R}^m$. In Appendix, we
give more details about our reconstructed space and the computer
implementation.

We next focus on the approximation property of the operator $\mc{R}^m$.
For element $K$, we define a constant 
\begin{displaymath}
  \Lambda(m, S(K)) = \max_{v \in \mb P_m(S(K))} \frac{\max_{\bm x \in
  S(K)} |v(\bm x)|}{\max_{\bm x \in \mc I_K} | v(\bm x)|}.
\end{displaymath}

We note that under some mild and practical conditions about $S(K)$,
the $\Lambda(m, S(K))$ has a uniform upper bound $\Lambda_m$, which
plays an important role in the approximation property analysis.  We
refer to \cite{li2012efficient, li2016discontinuous} for the
conditions and more details about the constant $\Lambda(m, S(K))$ and
the uniform upper bound. Besides, under such conditions the Lemma
\ref{le:appirr} could be generalized as
\begin{lemma}
  For any function $\bmr{q} \in \bmr{H}^{m+1}(S(K)) \cap H(\curl^0,
  S(K))$, there exists a constant $C$ such that there is a polynomial
  $\widetilde{\bmr{q}}_h \in \bmr{S}_m(S(K))$ such that 
  \begin{equation}
    \| \bmr{q} - \widetilde{\bmr{q}}_h\|_{L^2(S(K))} + h_K \| \nabla
    \left( \bmr{q} - \widetilde{\bmr{q}}_h \right) \|_{L^2(S(K))} \leq
    C h_K^{m+1} \|\bmr{q}\|_{ \bmr{H}^{m+1}(S(K))}.
    \label{eq:appirrpatch}
  \end{equation}
  \label{le:appirrpatch}
\end{lemma}
\begin{proof}
  It directly follows from \cite[Assumption A and Property M3]{li2016discontinuous}.
\end{proof}

With $\Lambda_m$, let us state the approximation property of the
operator $\mc{R}^m_K$.
\begin{theorem}
  Let $\bmr{f} \in \bmr{H}^{m+1}(\Omega) \cap H(\curl^0; \Omega)$ and
  $K \in \MTh$, there exists a constant $C$ such that 
  \begin{equation}
    \begin{aligned}
      \|\bmr{f} - \mc{R}^m_K \bmr{f} \|_{\bmr{H}^q(K)} &\leq C
      \Lambda_m h_K^{m+1-q} \|\bmr{f}\|_{\bmr{H}^{m+1}(S(K))}, \quad q
      = 0, 1, \\
      \| \nabla^q(\bmr{f} - \mc{R}^m_K \bmr{f}) \|_{L^2(\partial K)} &
      \leq C \Lambda_m h_K^{m+1 - q -1/2} \| \bmr{f} \|_{\bmr{H}^{m+1}
      (S(K))}, \quad q = 0, 1.\\
    \end{aligned}
    \label{eq:approximation}
  \end{equation}
  \label{th:approximation}
\end{theorem}
\begin{proof}
The estimates directly follows the proof of \cite[Lemma
2.4]{li2016discontinuous} and the Lemma \ref{le:appirrpatch}.
\end{proof}



\section{Sequential Least Squares Finite Element Approximation}
\label{sec:lsfem}
Let us define a new functional $\wt{J}_h(\cdot)$ by
\begin{equation}
  \begin{aligned}
    \wt{J}_h(\bmr{q}_h) \triangleq \sum_{K \in \MTh} \| \nabla \cdot
    \bmr{q}_h + f\|_{L^2(K)}^2 + &\sum_{e \in \MEh^i}
    \frac{1}{h} \| \jump{\bmr{q}_h \otimes \un} \|_{L^2(e)}^2 \\
    + &\sum_{e \in \MEh^b} \frac{1}{h} \|\bmr{q}_h \times \un -
    \nabla g \times \un\|_{L^2(e)}^2. \\
  \end{aligned}
  \label{eq:funcp}
\end{equation}
The terms in $\wt{J}_h(\bmr{q}_h)$ include the part related to the
flux in \eqref{eq:infdgep} and the term on boundary. We minimize this
functional in $\bmr{U}_h^m$ to have an approximate flux. The
corresponding minimization problem reads: {\it find
  $\bmr{p}_h \in \bmr{U}_h^m$ such that
\begin{equation}
  \wt{J}_h(\bmr{p}_h) = \inf_{\bmr{q}_h \in \bmr{U}_h^m} \wt{J}_h(
  \bmr{q}_h).
  \label{eq:discretepinf}
\end{equation}}
The Euler-Lagrange equation of this minimization problem is as: {\it
find $\bmr{p}_h \in \bmr{U}_h^m$ such that
\begin{equation}
  \wt{a}_h(\bmr{p}_h, \bmr{q}_h) = \wt{l}_h(\bmr{q}_h), \quad
  \bmr{q}_h \in \bmr{U}_h^m, 
  \label{eq:variational}
\end{equation}
where the bilinear form $\wt{a}_h(\cdot, \cdot)$ is
\begin{displaymath}
  \begin{aligned}
    \wt{a}_h(\bmr{p}_h, \bmr{q}_h) = \sum_{K \in \MTh} \int_K \nabla
    \cdot \bmr{p}_h \nabla \cdot \bmr{q}_h \d{x} &+ \sum_{e \in
    \MEh^i} \int_e \frac{1}{h} \jump{\bmr{p}_h \otimes \un} \jump{\bmr{q}_h
    \otimes \un} \d{s} \\
    &+ \sum_{e \in \MEh^b} \int_e \frac{1}{h} (\bmr{p}_h \times \un) \cdot
    (\bmr{q}_h \times \un) \d{s}, \\
  \end{aligned}
\end{displaymath}
and the linear form $\wt{l}_h(\cdot)$ is 
\begin{displaymath}
  \wt{l}_h(\bmr{q}_h) = \sum_{K \in \MTh} \int_K f \nabla \cdot \bmr{q}_h
  \d{x} + \sum_{e \in \MEh^b} \int_e \frac{1}{h} (\bmr{p}_h \times \un)
  \cdot (\nabla g \times \un) \d{s}.
\end{displaymath}}
Let
\begin{equation}
  \begin{aligned}
    \enorm{\bmr{q}_h}_{\bmr{p}}^2 \triangleq \sum_{K \in \MTh} \|
    \nabla \cdot \bmr{q}_h\|_{L^2(K)}^2 + \sum_{e \in \MEh^i}
    \frac{1}{h} \| \jump{\bmr{q}_h \otimes \un} \|_{L^2(e)}^2 +
    \sum_{e \in \MEh^b} \frac{1}{h} \|\bmr{q}_h \times
    \un\|_{L^2(e)}^2\\
  \end{aligned}
  \label{eq:pnorm}
\end{equation}
for
$\forall \bmr{q}_h \in \bmr{U}_h^m + \bmr{H}^1(\Omega) \cap H(\curl^0;
\Omega)$. The following lemma shows that $\enorm{\cdot}_{\bmr{p}}$
actually defines a norm on the space
$\bmr{U}_h^m + \bmr{H}^1(\Omega) \cap H(\curl^0; \Omega)$, referred as
the {\it energy norm} later on.
\begin{lemma}
  For any $\bmr{q}_h \in \bmr{U}_h^m + \bmr{H}^1(\Omega) \cap
  H(\curl^0; \Omega)$, there exists a constant $C$  such that 
  \begin{equation}
    \|\bmr{q}_h\|_{L^2(\Omega)} \leq C \enorm{\bmr{q}_h}_{\bmr{p}}.
    \label{eq:pisnorm}
  \end{equation}
  \label{le:pisnorm}
\end{lemma}
\begin{proof}
  The idea follows \cite[Lemma 1]{Bensow2005discontinuous} to apply
  the orthogonal decomposition of $\bmr{L}^2(\Omega)$. We only proof
  for the case $d = 2$ and it is almost trivial to extend the result
  for three dimensional case. Since $\bmr{q}_h \in \bmr{L}^2(\Omega)$,
  we let $\phi \in H^1(\Omega) \backslash \mb R$ be the only solution
  of
  \begin{displaymath}
    (\nabla \times \phi, \nabla \times \chi) = (\bmr{q}_h, \nabla
    \times \chi), \quad \forall \chi \in H^1(\Omega).
  \end{displaymath}
  This solution $\phi$ satisfies 
  \begin{displaymath}
    -\Delta \phi = \nabla \times \bmr{q_h}, \quad \text{in}\
    H^{-1}(\Omega).
  \end{displaymath}
  Applying the Green's formula, we have
  \begin{displaymath}
    0 = (\bmr{q}_h - \nabla \times \phi, \nabla \times \chi) = \left(
    (\bmr{q}_h - \nabla \times\phi) \times \un, \chi
      \right)_{L^2(\partial \Omega)}, \quad \forall \chi \in
      H^1(\Omega). 
  \end{displaymath}
  Thus there exists $v \in H^1_0(\Omega)$ such that $\nabla v =
  \bmr{q}_h - \nabla \times \phi$ \cite{girault1986finite}. Besides we
  have the stability estimates 
  \begin{equation}
    \|\chi\|_{H^1(\Omega)} \leq C \|\bmr{q}_h\|_{\bmr{L}^2(\Omega)} , 
    \quad 
    \|v\|_{H^1(\Omega)} \leq C \|\bmr{q}_h\|_{\bmr{L}^2(\Omega)}. 
    \label{eq:chiqstability}
  \end{equation}
  Further, we use the decomposition to obtain 
  \begin{displaymath}
    \begin{aligned}
      \|\bmr{q}_h\|_{\bmr{L}^2(\Omega)}^2 &= \left( \sum_{K \in \MTh}
      \int_K \bmr{q}_h \cdot \nabla v \d{x} + \int_K \bmr{q}_h \cdot
      \nabla \times \chi \d{x} \right) \\
      &= \sum_{K \in \MTh} \left( \int_{\partial K} v \bmr{q}_h \cdot
      \un \d{s} - \int_K v \nabla \cdot \bmr{q}_h \d{x} +
      \int_{\partial K} \chi \bmr{q}_h \times \un \d{s} \right). \\ 
      &= \sum_{e \in \MEh^i} \int_e \left( v \jump{\bmr{q}_h \cdot
      \un} + \chi \jump{\bmr{q}_h \times \un} \right) \d{s} + \sum_{e
      \in \MEh^b} \int_e \chi \bmr{q}_h \times \un \d{s} - \int_{K} v
      \nabla \cdot \bmr{q}_h \d{x}\\
    \end{aligned}
  \end{displaymath}
  And we have that 
  \begin{displaymath}
    \sum_{e \in \MEh^i} \int_e \left( \|\jump{ \bmr{q}_h \cdot \un}
    \|_{L^2(e)}^2 + \| \jump{ \bmr{q}_h \times \un} \|_{L^2(e)}^2
    \right) \d{s} \leq C \sum_{e \in \MEh^i} \int_e \| \jump{
    \bmr{q}_h \otimes \un}\|_{L^2(e)}^2 \d{s}. 
  \end{displaymath}
  Using the Cauchy-Schwarz inequality, trace inequality
  \eqref{eq:traceinequality} and the stability estimate
  \eqref{eq:chiqstability} could yield the estimate
  \eqref{eq:pisnorm}, which completes the proof.
\end{proof}
Since for
$\forall \bmr{q}_h \in \bmr{U}_h^m + \bmr{H}^1(\Omega) \cap H(\curl^0;
\Omega)$ we have
$\wt{a}_h(\bmr{q}_h, \bmr{q}_h) = \enorm{\bmr{q}_h}_{\bmr{p}}^2$, it
is implied that the problem \eqref{eq:variational} has a unique
solution. Moreover, we could establish the convergence result with
respect to the norm $\enorm{\cdot}_{\bmr{p}}$.
\begin{theorem}
  Let the solution
  $\bmr{p} \in \bmr{H}^{m+1}(\Omega) \cap H(\curl^0; \Omega)$ and let
  $\bmr{p}_h \in \bmr{U}_h^m$ be the solution to
  \eqref{eq:variational}, then we have
  \begin{equation}
    \enorm{\bmr{p} - \bmr{p}_h}_{\bmr{p}} \leq Ch^m
    \|\bmr{p}\|_{\bmr{H}^{m+1}(\Omega)}.
    \label{eq:pconvergence}
  \end{equation}
  \label{th:pconvergence}
\end{theorem}
\begin{proof}
  Since $\bmr{p}_h$ minimizes the problem \eqref{eq:discretepinf} and
  $\jump{\bmr{p} \otimes \un} = 0$, we have
  \begin{displaymath}
    \enorm{\bmr{p} - \bmr{p}_h}_{\bmr{p}}^2 = \wt{J}_h(\bmr{p}_h) \leq
    \wt{J}_h(\mc R^m \bmr{p}) = \enorm{\bmr{p} - \mc R^m
    \bmr{p}}_{\bmr{p}}^2.
  \end{displaymath}
  Therefore, we only need to bound
  $\enorm{\bmr{p} - \mc R^m \bmr{p}}_{\bmr{p}}$.

  By the approximation \eqref{eq:approximation} and trace inequality
  \eqref{eq:traceinequality}, we obtain that for element $K$,
  \begin{displaymath}
    \|\nabla\cdot \bmr{p} - \nabla \cdot \mc R_m \bmr{p} \|_{L^2(K)}
    \leq Ch_K^m \|\bmr{p}\|_{\bmr{H}^{m+1}(S(K))}, 
  \end{displaymath}
  and 
  \begin{displaymath}
    \begin{aligned}
      \|(\bmr{p} - \mc R_m \bmr{p}) \otimes \un \|_{L^2(\partial K)}^2
      & \leq C \| \bmr{p} - \mc R_m \bmr{p} \|_{L^2(\partial K)}^2\\
      & \leq C (h_K^{-1} \|\bmr{p} - \mc R_m \bmr{p} \|_{L^2(K)}^2 + h_K
      \|\nabla( \bmr{p} - \mc R_m \bmr{p}) \|_{L^2(K)}^2 ) \\
      & \leq Ch_K^{2m+1} \|\bmr{p}\|_{\bmr{H}^{m+1}(S(K))}. \\
    \end{aligned}
  \end{displaymath}
  The inequality \eqref{eq:pconvergence} is concluded by summing over
  all elements in the partition, which completes the proof.
\end{proof}
After getting the numerical flux $\bmr{p}_h$, the next step is to plug
it into the functional \eqref{eq:infdgep} to calculate the pressure
$u$. We define the functional ${J}^u_h(\cdot)$ as below:
\begin{equation}
  {J}^u_h(v) \triangleq \sum_{K \in \MTh} \|\nabla v -
  \bmr{p}_h\|_{L^2(K)}^2 + \sum_{e \in \MEh^i} \frac{1}{h} \| \jump{v}
  \|_{L^2(e)}^2 + \sum_{e \in \MEh^b} \frac{1}{h} \| v -
  g\|_{L^2(e)}^2.
  \label{eq:infu}
\end{equation}
To get an approximation to $u$, one may solve the minimization problem
for the functional ${J}^u_h(\cdot)$ in a certain approximation space.
We note that it is very flexible to choose the approximation space for
$u$. For instance, one may use the discontinuous finite element space
$V_h^m$ or the patch reconstructed space proposed in
\cite{li2016discontinuous}. Here we solve the pressure $u$ with the
standard Lagrange finite element space, which is defined as
\begin{displaymath}
  \wh{V}_h^m \triangleq \left\{ v_h \in C(\Omega) \ |\ v_h|_K \in
  \mathbb P_m(K), \quad \forall K \in \MTh\right\}.
\end{displaymath}
Due to the continuity of the space $\wh{V}_h^m$, the functional
${J}^u_h(v)$ is simplified as
\begin{equation}
  {J}^u_h(v) = \sum_{K \in \MTh} \|\nabla v -
  \bmr{p}_h\|_{L^2(K)}^2 + \sum_{e \in \MEh^b} \frac{1}{h} \| v -
  g\|_{L^2(e)}^2, \quad \forall v \in H^1(\Omega).
  \label{eq:infucfem}
\end{equation}
The following minimization problem gives the numerical solution to the
pressure $u$ in $\wh{V}_h^m$:
\begin{displaymath}
  \min_{v_h \in \wh{V}_h^m} {J}^u_h(v_h).
\end{displaymath}
The discrete variational problem equivalent to the minimization
problem reads: {\it find $u_h \in \wh{V}_h^m$ such that
\begin{equation}
  {a}^u_h(u_h, v_h) = {l}^u_h(v_h), \quad \forall v_h \in
  \wh{V}_h^m,
  \label{eq:discreteuinf}
\end{equation}
where the bilinear form ${a}^u_h(\cdot, \cdot)$ is given by
\begin{displaymath}
  {a}^u_h(u_h, v_h) = \sum_{K \in \MTh} \int_K \nabla u_h \cdot
  \nabla v_h \d{x} + \sum_{e \in \MEh^b} \int_e \frac{1}{h} u_h v_h
  \d{s},
\end{displaymath}
and the linear form ${l}^u_h(\cdot)$ is given by
\begin{displaymath}
  {l}^u_h = \sum_{K \in \MTh} \int_K \nabla v_h \cdot \bmr{p}_h \d{x}
  + \sum_{e \in \MEh^b} \int_e \frac{1}{h}v_h g \d{s}.
\end{displaymath}}
Analogous to the procedure we solve the flux $\bmr p$, we define
$\enorm{\cdot}_u$ as
\begin{displaymath}
  \enorm{v}_u^2 \triangleq \sum_{K \in \MTh} \| \nabla v\|_{L^2(K)}^2
  + \sum_{e \in \MEh^b} \frac{1}{h} \|v\|_{L^2(e)}^2, \quad \forall v
  \in H^1(\Omega).
\end{displaymath}
The inequality $\|v\|_{L^2(\Omega)} \leq C \enorm{v}_u$ \cite[Lemma
2.1]{arnold1982interior} ensures $\enorm{\cdot}_u$ is actually a norm
on $H^1(\Omega)$, which actually guarantees the unisolvability of the
problem \eqref{eq:discreteuinf}. $\enorm{\cdot}_u$ is referred as the
{\it energy norm} on $\wh{V}_h^m$ since now on. Further, the error
estimate with respect to $\enorm{\cdot}_u$ is given in the theorem
below as:
\begin{theorem}
  Let the solution $u \in H^{m+1}(\Omega)$ and let $u_h \in
  \wh{V}_h^m$ be the solution to \eqref{eq:discreteuinf}, then we have
  \begin{equation}
    \enorm{u - u_h}_u \leq C\|\bmr{p} - \bmr{p}_h
    \|_{\bmr{L}^2(\Omega)} + Ch^m \|u\|_{H^{m+1}(\Omega)},
    \label{eq:uconvergence}
  \end{equation}
  where $\bmr{p}_h$ is the solution to \eqref{eq:discretepinf}.
  \label{th:uconvergence}
\end{theorem}
\begin{proof}
  Let $u_I \in \wh{V}_h^m$ be the interpolant of $u$ and we have that
  \begin{displaymath}
    \begin{aligned}
      {J}^u_h(u_h) &\leq {J}^u_h(u_I) \\
      \enorm{u - u_h}_u^2 &\leq C({J}^u_h(u_I) + \|\nabla 
      u - \bmr{p}_h\|^2_{\bmr{L}^2(\Omega)}) \\
      &\leq C \left( \enorm{u - u_I}_u^2 + \|\bmr{p} - \bmr{p}_h
      \|_{L^2(\Omega)}^2\right). \\
    \end{aligned}
  \end{displaymath}
  The approximation property of the space $\wh{V}_h^m$ gives us
  \cite{ciarlet2002finite}: 
  \begin{displaymath}
    \enorm{u - u_I}_u \leq Ch^m\|u\|_{H^{m+1}(\Omega)},
  \end{displaymath}
  which yields the estimate \eqref{eq:uconvergence} and completes the
  proof.
\end{proof}
Then we can have the error estimate under $L^2$-norm:
\begin{theorem}
  Let the solution $u \in H^{m+1}(\Omega)$ and let $u_h
  \in \wh{V}_h^m$ be the solution to \eqref{eq:discreteuinf}, then we
  have
  \begin{equation}
    \|u - u_h\|_{L^2(\Omega)} \leq C_0 \| \bmr{p} -
    \bmr{p}_h\|_{L^2(\Omega)} + C_1h^{m+1}\|u\|_{H^{m+1}(\Omega)},
    \label{eq:ul2estiate}
  \end{equation}
  where $\bmr{p}_h$ is the solution to \eqref{eq:discretepinf}.
  \label{th:ul2estiate}
\end{theorem}
\begin{proof}
  Let $e_h = u - u_h$ and from the definition of ${a}^u_h(\cdot,
  \cdot)$, one see that 
  \begin{displaymath}
    {a}^u_h(e_h, v_h) = (\bmr{p} - \bmr{p}_h, v_h), \quad \forall v_h
    \in \wh{V}_h^m.
  \end{displaymath}
  We first show that
  $\|e_h\|_{H^{-1/2}(\partial \Omega)} \leq C_0h \enorm{e_h}_u + C_1h
  \|\bmr{p} - \bmr{p}_h\|_{L^2(\Omega)}$, where
  \begin{displaymath}
    \|e_h\|_{H^{-1/2}(\partial \Omega)} = \sup_{\tau \in
    H^{1/2}(\partial \Omega)} \frac{\left( e_h, \tau
    \right)_{L^2(\partial \Omega)}}{\|\tau\|_{H^{1/2}(\partial
    \Omega)}}.
  \end{displaymath}
  We let $\alpha \in H^1(\Omega)$ which solves $\Delta \alpha = 0$ in
  $\Omega$, $\alpha = \tau$ on $\partial \Omega$, and we let $\alpha_I
  \in \wh{V}_h^m$ be interpolant of $\alpha$. Then we have that 
  \begin{displaymath}
    \begin{aligned}
      (e_h, \tau)_{L^2(\partial \Omega)} &= h(h^{-1}(e_h,
      \alpha)_{L^2(\partial \Omega)}) \\
      &= h(h^{-1}(e_h, \alpha)_{L^2(\partial \Omega)} - {a}^u_h(e_h,
      \alpha_I)) + h
      (\bmr{p} - \bmr{p}_h, \alpha_I) \\
      &= h(h^{-1}(e_h, \alpha - \alpha_I)_{L^2(\partial \Omega)} -
      (\nabla e_h, \nabla \alpha_I)_{L^2(\Omega)}) + h (\bmr{p} -
      \bmr{p}_h, \alpha_I) \\
      &\leq C_0h \enorm{e_h}_u (\|h(\alpha - \alpha_I)\|_{L^2(\partial
        \Omega)} + \|\nabla \alpha_I\|_{L^2(\Omega)}) + C_1h \|\bmr{p}
      -
      \bmr{p}_h\|_{L^2(\Omega)} \|\alpha_I\|_{L^2(\Omega)} \\
      & \leq C_0h \enorm{u_h}_u \|\alpha\|_{H^1(\Omega)} + C_1h
      \|\bmr{p} - \bmr{p}_h\|_{L^2(\Omega)} \|\alpha\|_{H^1(\Omega)}.
      \\
    \end{aligned}
  \end{displaymath}
  We complete the proof by the regularity estimate
  $\|\alpha\|_{H^1(\Omega)} \leq C \|\tau\|_{H^{1/2}(\partial
  \Omega)}$. 

  Given $\psi \in L^2(\Omega)$ and we let $w \in H^2(\Omega)$ which
  solves $- \Delta w = \psi$ in $\Omega$, $w = 0$ on $\partial
  \Omega$. We denote by $w_I \in \wh{V}_h^m$ the interpolant of $w$.
  Then we could deduce that
  \begin{displaymath}
    \begin{aligned}
      (e_h, \psi) &= (\nabla e_h, \nabla w) - \left(e_h,
        \frac{\partial
          w}{\partial \un}\right)_{L^2(\partial \Omega)} \\
      &= {a}^u_h(e_h, w - w_I) + (\bmr{p} - \bmr{p}_h, w_I) -
      \left(e_h,
        \frac{\partial w}{\partial \un}\right)_{L^2(\partial \Omega)} \\
      & \leq Ch \enorm{e_h} \|w\|_{H^2(\Omega)} + \|\bmr{p} -
      \bmr{p}_h \|_{\bmr{L}^2(\Omega)} \|w\|_{H^2(\Omega)} +
      \|e_h\|_{H^{-1/2}(\partial \Omega)} \left\|\frac{\partial
          w}{\partial \un} \right\|_{H^{1/2}(\partial \Omega)}.
    \end{aligned}
  \end{displaymath}
  Let $\psi = e_h$, and combining the bound of
  $\|e_h\|_{H^{-1/2}(\Omega)}$, the regularity estimate
  $\|w\|_{H^2(\Omega)} \leq C \|\psi\|_{L^2(\Omega)}$ and the
  approximation property of $\enorm{e_h}_u$ could yield the estimate
  \eqref{eq:ul2estiate}, which completes the proof.
\end{proof}
\begin{remark}
  Until now the method we established is only for the problem with the
  Dirichlet boundary condition. For the Neumann boundary condition
  $\nabla u \cdot \un = g$ on $\partial \Omega$, the boundary term in
  \eqref{eq:funcp} and \eqref{eq:infucfem} should be modified as
  \begin{displaymath}
    \sum_{e \in \MEh^b} \frac{1}{h} \|\bmr{q}_h \cdot \un -
    g\|_{L^2(e)}^2 \qquad \text{and} \qquad \sum_{e \in \MEh^b} \frac{1}{h} \left\|
      \frac{\partial v}{\partial \un} - g\right\|_{L^2(e)}^2,
  \end{displaymath}
  respectively. It is almost trivial to extend our method in this
  section to the problem with the Neumann boundary condition.
\end{remark}


\section{Numerical Results}
\label{sec:numericalresults}
In this section, we conduct some numerical experiments to show the
accuracy and efficiency of the proposed method in
Section~\ref{sec:lsfem}. For simplicity, we select the cardinality $\#
S(K)$ uniformly and we list a group of reference values of $\# S(K)$
for different $m$ in Tab.~\ref{tab:numSK}. 

\begin{table}
  \centering
  \renewcommand\arraystretch{1.3}
  \begin{tabular}{p{2.0cm}| p{2.0cm}|p{1.5cm}|p{1.5cm}|p{1.5cm} }
    \hline\hline
    \multirow{2}{*}{$d = 2$} &  $m$ & 1 & 2 & 3 \\
    \cline{2-5}
     & $\# S(K)$ & 6 & 10 & 16 \\
    \hline
    \multirow{2}{*}{$d = 3$} &  $m$ & 1 & 2 & 3 \\
    \cline{2-5}
     & $\# S(K)$ & 8 & 15 & 25 \\
     \hline \hline
  \end{tabular}
  \caption{$\# S(K)$ for $1 \leq m \leq 3$.}
  \label{tab:numSK}
\end{table}

\subsection{Convergence order study}
We first examine the numerical convergence to verify the theoretical
prediction and exhibit the flexibility of our method.

\paragraph{\bf Example 1} We first consider a two-dimensional Poisson
problem with Dirichlet boundary condition on the domain $\Omega = [0,
1] \times [0, 1]$. The exact solution $u(x, y)$ is taken as
\begin{displaymath}
  u(x, y) = \sin(2\pi x) \sin(4\pi y), 
\end{displaymath}
and the source term $f$ and the boundary data $g$ are chosen
accordingly.

We solve this problem on a series of triangular meshes (see
Fig.~\ref{fig:triangulation}) with mesh size $h = 1/10$, $1/20$,
$c\dots$, $1/80$ and we first use the space pairs $\bmr{U}_h^m \times
\wh{V}_h^m(1 \leq m \leq 3)$ to solve the flux and pressure. In this
setting, from \eqref{eq:ul2estiate} we could see that the optimal
convergence order of $u_h$ depends on the convergence rate of
$\|\bmr{p} - \bmr{p}_h\|_{\bmr{L}^2(\Omega)}$. Although, we can not
develop a theoretical verification for the optimal convergence of
$\bmr{p}_h$ under $L^2$ norm, the computed convergence rates of
$\|\bmr{p} - \bmr{p}_h\|_{\bmr{L}^2(\Omega)}$ seem optimal for odd
$m$.  The $L^2$ norm and the energy norm of the errors in the
approximation to the exact solution are gathered in
Tab.~\ref{tab:ex1errorm1}. We could observe that for odd $m$, the
errors $\|u -u_h\|_{L^2(\Omega)}$, $\enorm{ u - u_h}_u$,  $\|\bmr{p} -
\bmr{p}_h\|_{\bmr{L}^2(\Omega)}$ and $\enorm{\bmr{p} -
\bmr{p}_h}_{\bmr{p}}$ converge to zero optimally as the mesh is
refined. For even $m$, the orders of convergence under $L^2$-norm are
suboptimal. Moreover, from the estimate \eqref{eq:ul2estiate} one
could observe that if we decrease the space approximating pressure by
one order, we could obtain the optimality for $u$ approximations.  The
errors with the space pairs $\bmr{U}_h^m \times \wh{V}_h^{m-1}$ are
collected in Tab.~\ref{tab:ex1errorm2}, which clearly shows the
optimal convergence of $u_h$ for both measurements.  Besides, we note
that all the convergence rates are consistence with the theoretical
predictions.

\begin{figure}[!htp]
  \centering
  \includegraphics[width=0.4\textwidth]{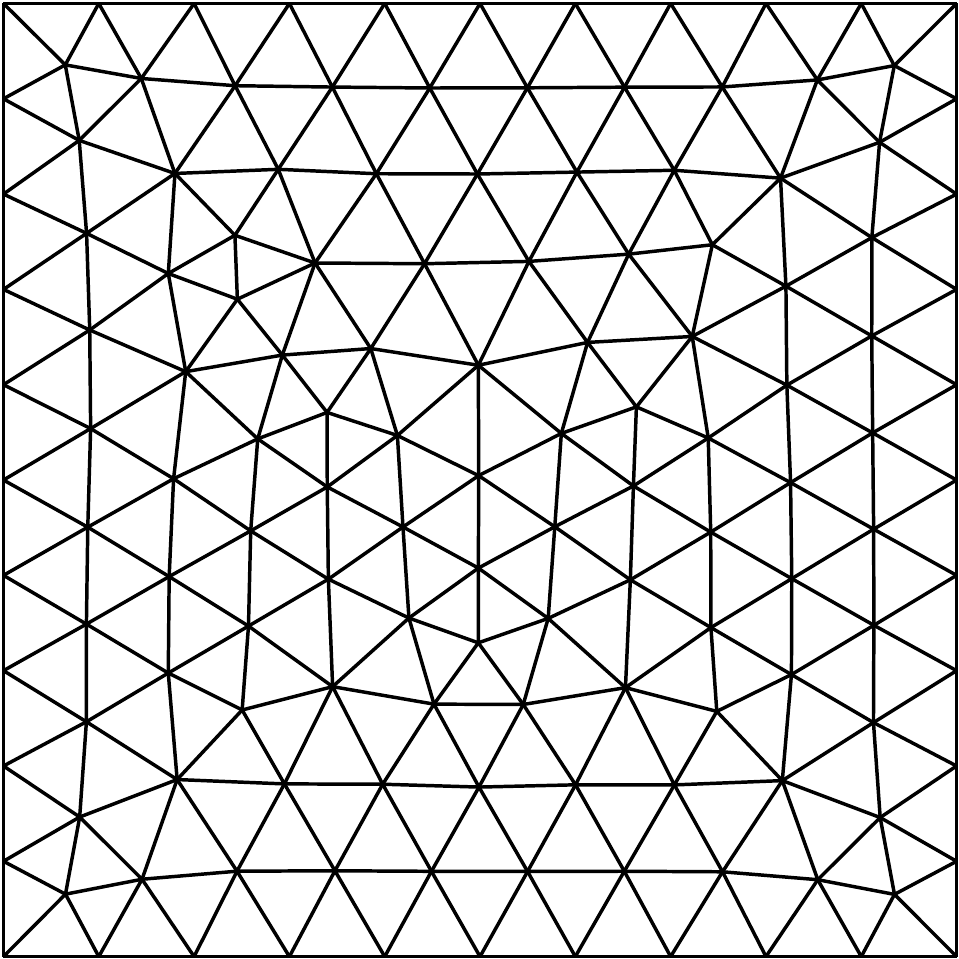}
  \hspace{25pt}
  \includegraphics[width=0.4008\textwidth]{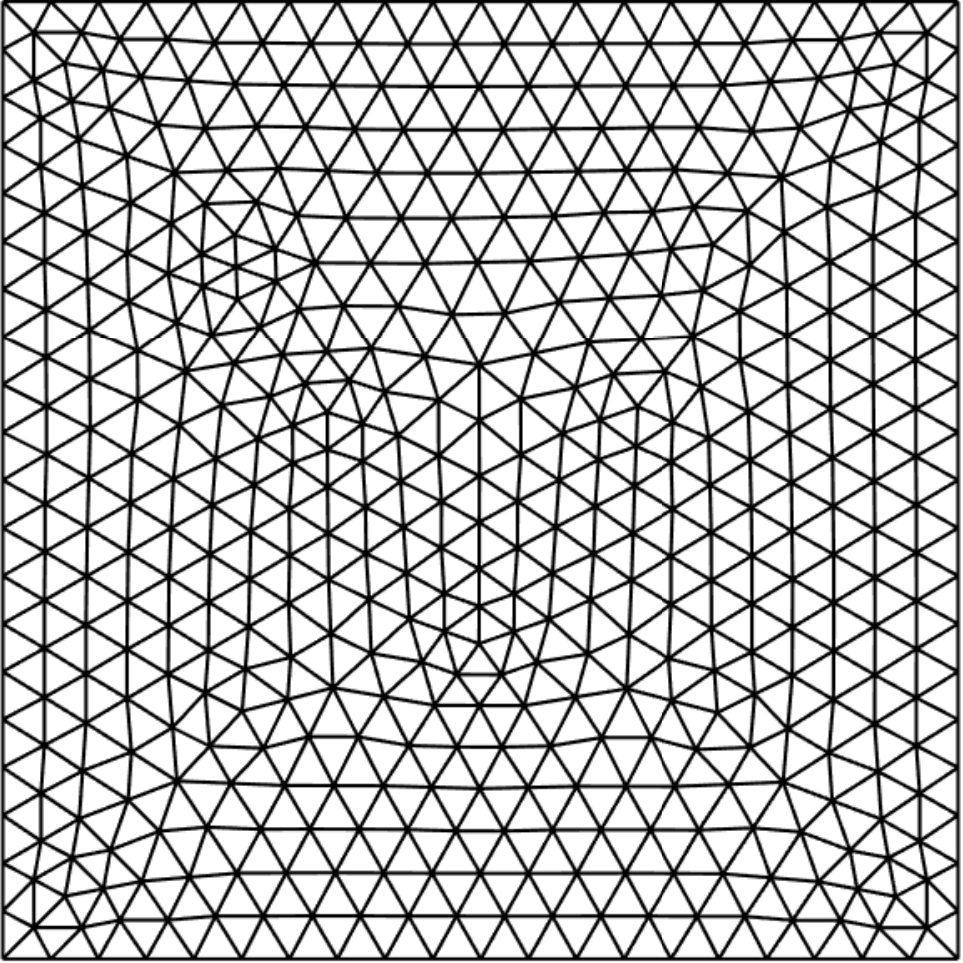}
  \caption{The triangular meshes with mesh size $h = 1/10$ (left) and
  $h = 1/20$ for Example 1.}
  \label{fig:triangulation}
\end{figure}

\begin{table}
  \centering
  \renewcommand\arraystretch{1.2}
  \begin{tabular}{p{0.3cm}| p{0.2cm} p{1.8cm} p{0.8cm} p{1.8cm}
    p{0.8cm} p{1.8cm} p{0.8cm} p{1.8cm} p{0.8cm}}
    \hline\hline
    $m$ & & $\|e_u\|_{L^2(\Omega)} $ & order & $\enorm{e_u}_u$ & order
    & $\| e_{\bmr{p}} \|_{\bmr{L}^2(\Omega)} $ & order &
    $\enorm{e_{\bmr{p}}}_{\bmr{p}}$ & order \\
    \hline
   \multirow{5}{*}{$1$} 
   & 1 & 1.0602e-01 & -    & 2.5550e-00 & -    
       & 1.1553e-00 & -    & 2.9109e+01 & -    \\
   & 2 & 3.0872e-02 & 1.80 & 1.2677e-00 & 1.01 
       & 3.3347e-01 & 1.80 & 1.5319e+01 & 0.93 \\
   & 3 & 8.3590e-03 & 1.90 & 6.3053e-01 & 1.01 
       & 8.7712e-02 & 1.90 & 7.9176e+00 & 0.95 \\
   & 4 & 2.1548e-03 & 1.96 & 3.1463e-01 & 1.00 
       & 2.2647e-02 & 1.96 & 4.0133e+00 & 0.98 \\
   & 5 & 5.4473e-04 & 1.98 & 1.5723e-02 & 1.00 
       & 5.7033e-03 & 1.98 & 2.0137e+00 & 1.00 \\
    \hline
    \multirow{5}{*}{$2$}
   & 1 & 5.5862e-02 & -    & 9.3425e-01 & -    
       & 9.1461e-01 & -    & 8.0168e+00 & - \\
   & 2 & 1.8898e-02 & 1.57 & 2.8628e-01 & 1.71 
       & 2.7402e-01 & 1.73 & 1.7807e+00 & 2.17 \\
   & 3 & 4.9746e-03 & 1.93 & 7.3469e-02 & 1.93 
       & 7.1190e-02 & 1.95 & 4.1888e-01 & 2.08 \\
   & 4 & 1.2538e-03 & 1.99 & 1.8776e-02 & 1.98 
       & 1.8016e-02 & 1.98 & 1.0111e-01 & 2.03 \\
   & 5 & 3.1393e-04 & 2.00 & 4.7137e-03 & 1.99 
       & 4.5126e-03 & 2.00 & 2.4633e-02 & 2.02 \\
   \hline
   \multirow{5}{*}{$3$}
   & 1 & 5.2485e-03 & -    & 1.6872e-01 & -    
       & 1.2492e-01 & -    & 3.7196e+00 & - \\
   & 2 & 3.9516e-04 & 3.73 & 1.9952e-02 & 3.07 
       & 9.2700e-03 & 3.75 & 4.6565e-01 & 2.95 \\
   & 3 & 2.1869e-05 & 4.17 & 2.0437e-03 & 3.28 
       & 5.9833e-04 & 3.95 & 6.0447e-02 & 2.97 \\
   & 4 & 1.1300e-06 & 4.27 & 2.2652e-04 & 3.17 
       & 3.8808e-05 & 3.95 & 7.7175e-03 & 2.97 \\
   & 5 & 6.0716e-08 & 4.07 & 2.7352e-05 & 3.05 
       & 2.4584e-06 & 3.98 & 9.7343e-04 & 2.99 \\
   \hline\hline
  \end{tabular}
  \caption{Example 1. The errors $e_u = u - u_h, e_{\bmr{p}} = \bmr{p}
  - \bmr{p}_h$, and the orders of convergence with the spaces
  $\bmr{U}_h^m \times \wh{V}_h^m(1 \leq m \leq 3)$.}
  \label{tab:ex1errorm1}
\end{table}

\begin{table}
  \centering
  \renewcommand\arraystretch{1.2}
  \begin{tabular}{p{0.3cm}| p{0.2cm} p{1.8cm} p{0.8cm} p{1.8cm}
    p{0.8cm} p{1.8cm} p{0.8cm} p{1.8cm} p{0.8cm}}
    \hline\hline
    $m$ & & $\|e_u\|_{L^2(\Omega)} $ & order & $\enorm{e_u}_u$ & order
    & $\| e_{\bmr{p}} \|_{\bmr{L}^2(\Omega)} $ & order &
    $\enorm{e_{\bmr{p}}}_{\bmr{p}}$ & order \\
    \hline
    \multirow{5}{*}{$2$}
   & 1 & 6.8296e-02 & -    & 2.4561e-00 & -    
       & 9.1461e-01 & -    & 8.0168e+00 & - \\
   & 2 & 1.7533e-02 & 1.96 & 1.2484e-00 & 0.97 
       & 2.7402e-01 & 1.73 & 1.7807e+00 & 2.17 \\
   & 3 & 4.4126e-03 & 1.99 & 6.2683e-01 & 0.99 
       & 7.1190e-02 & 1.95 & 4.1888e-01 & 2.08 \\
   & 4 & 1.1050e-03 & 2.00 & 3.3137e-01 & 1.00 
       & 1.8016e-02 & 1.98 & 1.0111e-01 & 2.03 \\
   & 5 & 2.7636e-04 & 2.00 & 1.5691e-01 & 1.00 
       & 4.5126e-03 & 2.00 & 2.4633e-02 & 2.02 \\
   \hline
   \multirow{5}{*}{$3$}
   & 1 & 4.9662e-03 & -    & 3.8263e-01 & -    
       & 1.2492e-01 & -    & 3.7196e+00 & - \\
   & 2 & 6.3248e-04 & 2.97 & 9.7317e-02 & 1.97 
       & 9.2700e-03 & 3.75 & 4.6565e-01 & 2.95 \\
   & 3 & 7.9437e-05 & 2.99 & 2.4434e-02 & 1.99 
       & 5.9833e-04 & 3.95 & 6.0447e-02 & 2.97 \\
   & 4 & 9.9415e-06 & 3.00 & 6.1151e-03 & 2.00 
       & 3.8808e-05 & 3.95 & 7.7175e-03 & 2.97 \\
   & 5 & 1.2430e-06 & 3.00 & 1.5291e-03 & 2.00 
       & 2.4584e-06 & 3.98 & 9.7343e-04 & 2.99 \\
   \hline\hline
  \end{tabular}
  \caption{Example 1. The errors $e_u = u - u_h, e_{\bmr{p}} = \bmr{p}
  - \bmr{p}_h$, and the orders of convergence with the spaces
  $\bmr{U}_h^{m} \times \wh{V}_h^{m-1}(2 \leq m \leq 3)$.}
  \label{tab:ex1errorm2}
\end{table}

\paragraph{\bf Example 2} In this example, we consider the sample
problem as in Example 1. But we use a sequence of polygonal meshes
consisting of elements with various geometries (see Fig.
\ref{fig:voroni}), which are generated by {\tt PolyMesher}
\cite{talischi2012polymesher}. We only solve the flux, and we present
the corresponding errors in the energy norm and $L^2$ norm and their
respective computed rates in Tab. ~\ref{tab:ex2error1}.  Again we
observe the optimal convergence for both norms when $m$ is odd. For
even $m$, $\|\bmr{p} - \bmr{p}_h \|_{\bmr{L}^2(\Omega)}$ tends to zero
in a suboptimal way. To apply the method on meshes with different
geometry, it is an advantage inherited from the DG method. On such
meshes, the convergence order is agreed with our error estimates again.

\begin{figure}[!htp]
  \centering
  \includegraphics[width=0.4\textwidth]{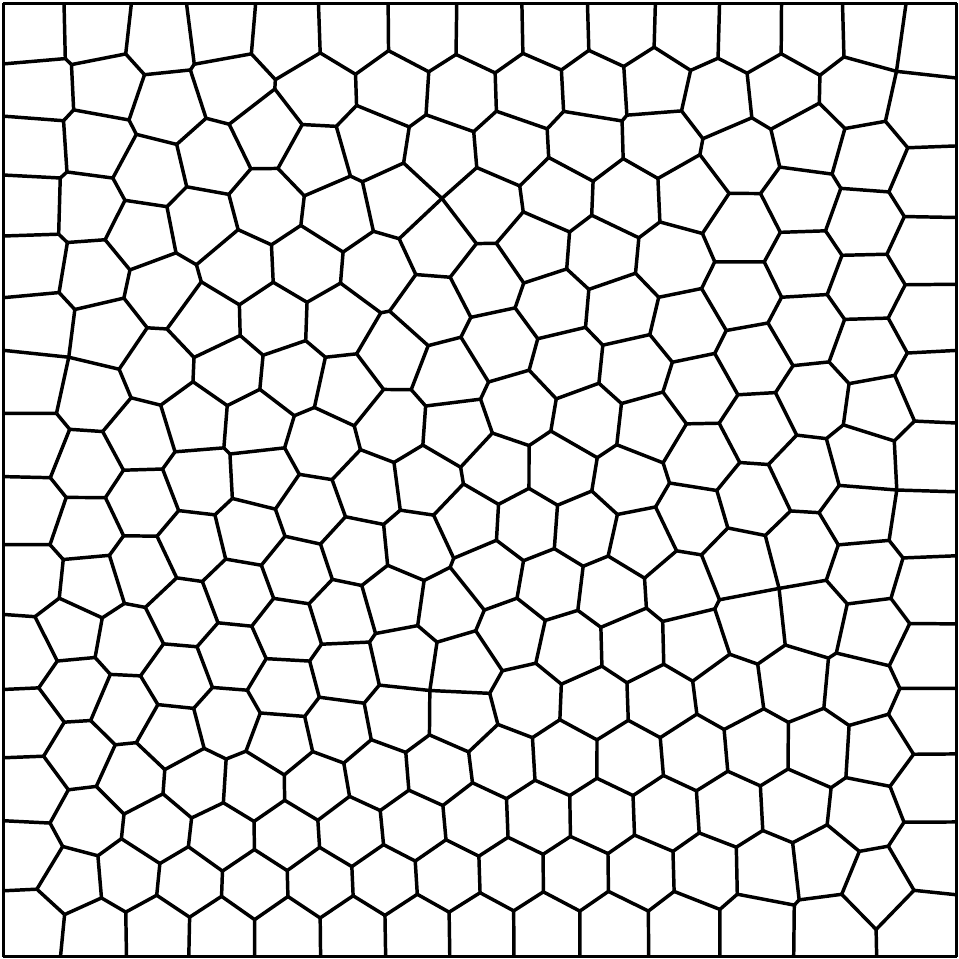}
  \hspace{25pt}
  \includegraphics[width=0.4\textwidth]{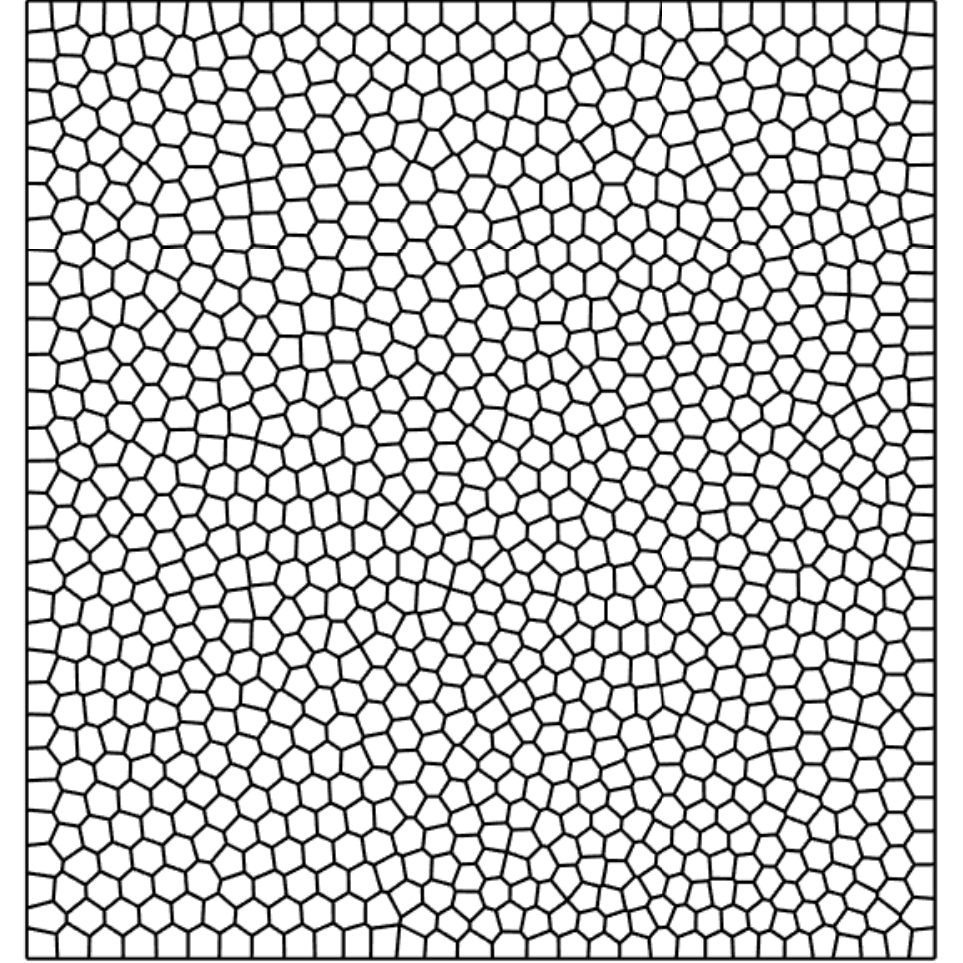}
  \caption{The polygonal meshes with 250 elements (left) / 1000
  elements (right).}
  \label{fig:voroni}
\end{figure}

\begin{table}
  \centering
  \begin{tabular}{p{0.3cm}| p{1.5cm} p{2.5cm} p{2cm} p{2.5cm} p{2cm}}
    \hline\hline
    $m$ & DOFs & $\|e_{\bmr{p}} \|_{\bmr{L}^2(\Omega)}$ & order & $\|
    e_{\bmr{p}}\|_{\bmr{p}}$ & order \\
    \hline
    \multirow{4}{*}{$1$} 
    & 500    & 1.0485e-00 & -    & 2.6456e+01 & -    \\
    & 2000   & 2.7316e-01 & 1.94 & 1.3244e+01 & 0.99 \\
    & 8000   & 6.5948e-02 & 2.05 & 6.5998e-00 & 1.00 \\
    & 32000  & 1.6203e-02 & 2.03 & 3.2658e-00 & 1.01 \\
    \hline
    \multirow{4}{*}{$2$} 
    & 500    & 4.4773e-01 & -    & 6.1493e-00 & -    \\
    & 2000   & 1.2630e-01 & 1.83 & 1.3713e-00 & 2.16 \\
    & 8000   & 3.0209e-02 & 2.06 & 3.3353e-01 & 2.03 \\
    & 32000  & 7.4860e-03 & 2.01 & 8.2873e-02 & 2.01 \\
    \hline
    \multirow{5}{*}{$3$} 
    & 500    & 1.6412e-01 & -    & 4.5508e-00 & -    \\ 
    & 2000   & 1.0449e-02 & 3.97 & 6.2226e-01 & 2.88 \\
    & 8000   & 6.3315e-04 & 4.05 & 8.1210e-02 & 2.95 \\
    & 32000  & 3.8188e-05 & 4.03 & 1.0205e-02 & 2.99 \\
    \hline\hline
  \end{tabular}
  \caption{Example 2. The errors $e_{\bmr{p}} = \bmr{p} - \bmr{p}_h$,
  and the orders of convergence with the spaces $\bmr{U}_h^m(1 \leq m
  \leq
  3)$.}
  \label{tab:ex2error1}
\end{table}

\paragraph{\bf Example 3} In this example, we consider the mild wave
front problem, which is the Poisson equation on the unit square with
Dirichlet boundary conditions. The data functions $f$ and $g$ are
selected such that the exact solution is 
\begin{displaymath} u(x, y) =
  \arctan(\alpha(r - r_0)), \quad (x, y) \in [0, 1]^2,
\end{displaymath} 
where $r = \sqrt{(x - x_0)^2 + (y - y_0)^2}$. The
mild wave front uses $(x_0, y_0) = (-0.05, -0.05)$, $r_0 = 0.7$,
$\alpha = 10$ and it is a problem of near singularities.  For this
problem, the high-order accuracy is preferred \cite{Mitchell2015high}.
We use a sequence of quasi-uniform triangular meshes (see Fig.
\ref{fig:ex3triangulation}) and we solve the problem with spaces
$\bmr{U}_h^m \times \wh{V}_h^m(1 \leq m \leq 3)$. We list the errors
in approximation to $\bmr{p}$ and $u$ in Tab. ~\ref{tab:ex3errorm1}.
It is clear that the proposed method yields the same convergence rates
as the Example 1, which validates our theoretical estimates.

\begin{figure}[!htp]
  \centering
  \includegraphics[width=0.4\textwidth]{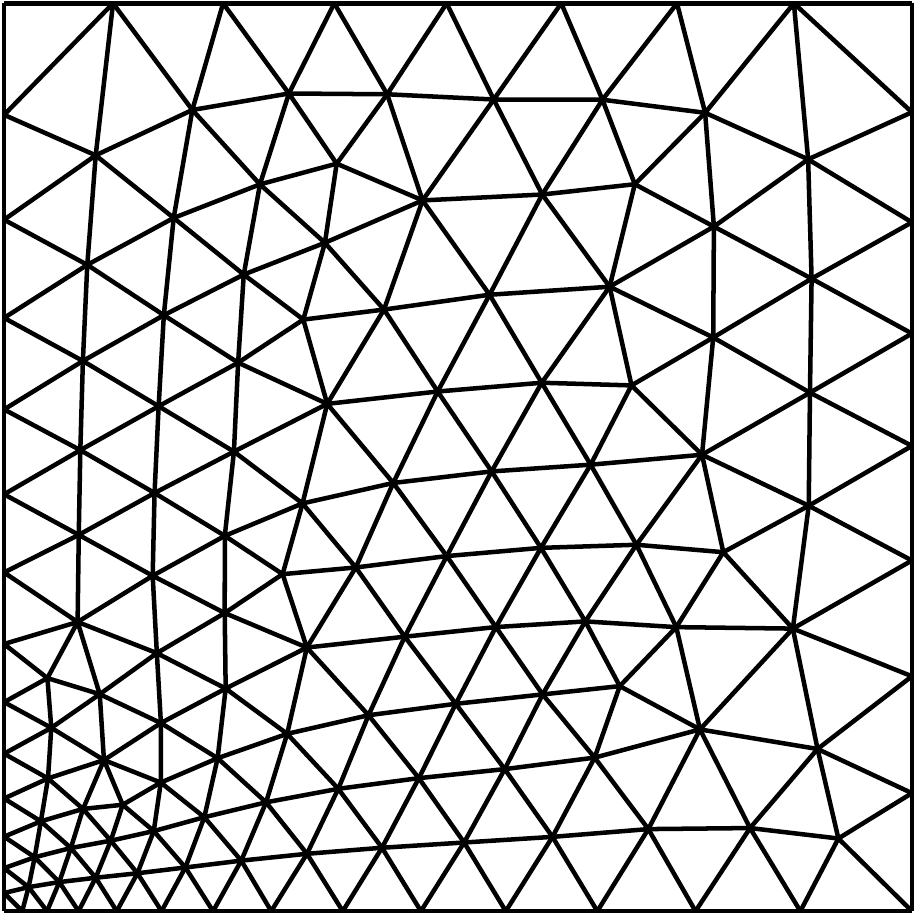}
  \hspace{25pt}
  \includegraphics[width=0.4\textwidth]{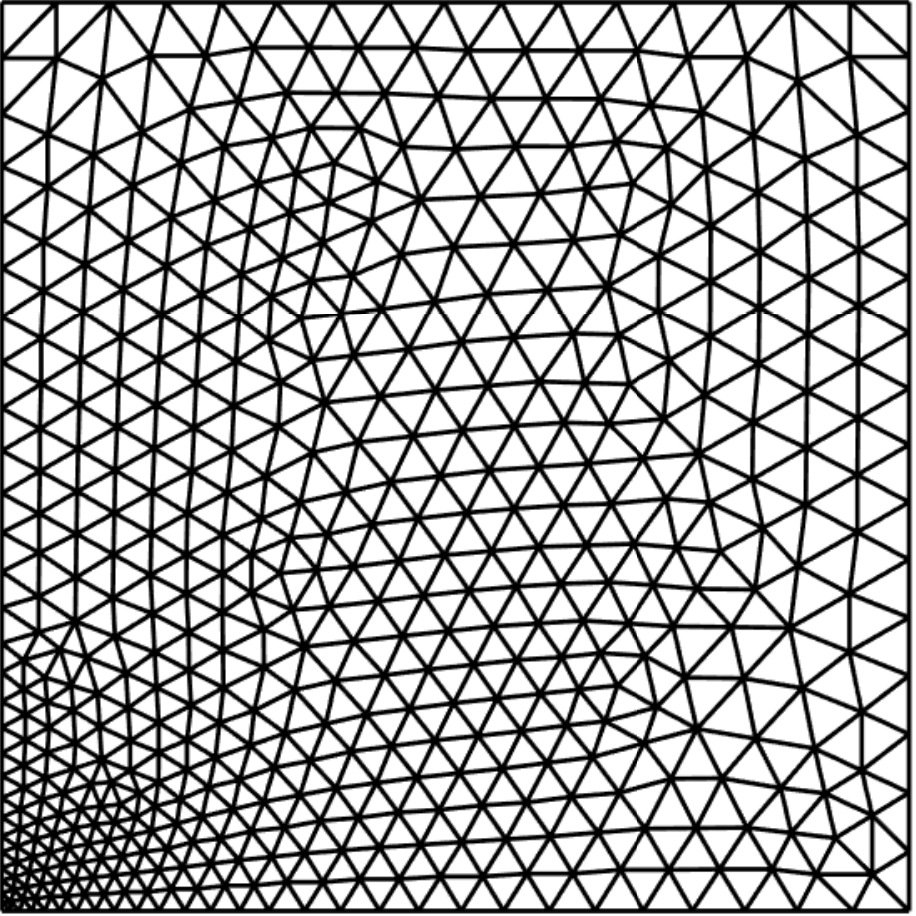}
  \caption{The triangular meshes with 246 elements (left) and
  984 elements (right) for Example 3.}
  \label{fig:ex3triangulation}
\end{figure}

\begin{table}
  \centering
  \renewcommand\arraystretch{1.2}
  \begin{tabular}{p{0.3cm}| p{0.2cm} p{1.8cm} p{0.8cm} p{1.8cm}
    p{0.8cm} p{1.8cm} p{0.8cm} p{1.8cm} p{0.8cm}}
    \hline\hline
    $m$ & & $\|e_u\|_{L^2(\Omega)} $ & order & $\enorm{e_u}_u$ & order
    & $\| e_{\bmr{p}} \|_{\bmr{L}^2(\Omega)} $ & order &
    $\enorm{e_{\bmr{p}}}_{\bmr{p}}$ & order \\
    \hline
   \multirow{5}{*}{$1$} 
   & 1 & 4.3807e-02 & -    & 4.9822e-01 & -    
       & 1.1553e-00 & -    & 1.0256e+01 & -    \\
   & 2 & 1.6473e-02 & 1.41 & 4.0917e-01 & 1.03 
       & 3.3347e-01 & 1.80 & 5.3347e+00 & 0.95 \\
   & 3 & 3.5661e-03 & 2.21 & 1.9515e-01 & 1.07 
       & 8.7712e-02 & 1.90 & 2.6486e+00 & 1.01 \\
   & 4 & 8.6682e-04 & 2.03 & 9.5962e-02 & 1.02 
       & 2.2647e-02 & 1.96 & 1.3231e+00 & 1.00 \\
   & 5 & 2.1263e-04 & 2.03 & 4.7761e-02 & 1.00 
       & 5.7033e-03 & 1.98 & 6.6057e-01 & 1.00 \\
    \hline
    \multirow{5}{*}{$2$}
   & 1 & 1.5200e-02 & -    & 2.9032e-01 & -    
       & 2.4411e-01 & -    & 5.6918e+00 & - \\
   & 2 & 5.3703e-03 & 1.51 & 9.0132e-02 & 1.68 
       & 8.9263e-02 & 1.45 & 1.3870e+00 & 2.03 \\
   & 3 & 1.4510e-03 & 1.89 & 2.5011e-02 & 1.85 
       & 2.5413e-02 & 1.82 & 3.1295e-01 & 2.10 \\
   & 4 & 3.6778e-04 & 1.98 & 6.5013e-02 & 2.00 
       & 6.7113e-03 & 1.92 & 7.1999e-02 & 2.11 \\
   & 5 & 9.1211e-05 & 2.01 & 1.6380e-03 & 1.98 
       & 1.6989e-03 & 1.99 & 1.7550e-02 & 2.03 \\
   \hline
   \multirow{5}{*}{$3$}
   & 1 & 1.0333e-02 & -    & 8.0091e-02 & -    
       & 2.0391e-01 & -    & 5.8500e+00 & - \\
   & 2 & 1.1023e-03 & 3.23 & 1.2076e-02 & 2.72 
       & 1.7701e-02 & 3.52 & 9.7265e-01 & 2.59 \\
   & 3 & 6.7612e-05 & 4.03 & 1.2368e-03 & 3.28 
       & 1.1398e-03 & 3.96 & 1.3999e-01 & 2.80 \\
   & 4 & 4.2528e-06 & 4.00 & 1.2956e-04 & 3.26 
       & 7.4761e-05 & 3.93 & 1.8073e-02 & 2.96 \\
   & 5 & 2.2322e-07 & 4.12 & 1.4319e-05 & 3.17 
       & 4.7259e-06 & 3.98 & 2.2425e-03 & 3.01 \\
   \hline\hline
  \end{tabular}
  \caption{Example 3. The errors $e_u = u - u_h, e_{\bmr{p}} = \bmr{p}
  - \bmr{p}_h$, and the orders of convergence with the spaces
  $\bmr{U}_h^m \times \wh{V}_h^m(1 \leq m \leq 3)$.}
  \label{tab:ex3errorm1}
\end{table}

\paragraph{\bf Example 4} In this example, we exhibit the performance
of the proposed method with the problem with a corner singularity. We
consider the L-shaped domain $\Omega = [-1, 1]^2 \backslash [0, 1)
\times (-1, 0]$ and we use a series of triangular meshes, see Fig.
\ref{fig:Lshapetriangulation}. Following
\cite{Mitchell2013collection}, we let the exact solution be 
\begin{displaymath}
  u(r, \theta) = r^{5/3} \sin(5\theta/3)
\end{displaymath}
in polar coordinate and impose the Dirichlet boundary condition. The
data $f$ and the function $g$ are chosen accordingly. We notice that
$u(r, \theta)$ only belongs to $H^{2 + s}$ with $s < 2/3$. In Tab.
\ref{tab:ex4error}, we list the errors measured in the energy norm and
$L^2$ norm for both flux and pressure. Here we observe that the error
$\enorm{\bmr{p} - \bmr{p}_h}_{\bmr{p}}$ decreases at the rate
$O(h^{2/3})$ which matches with the fact that $\bmr{p}$ only belongs
to $H^{5/3 - \varepsilon}(\Omega)$. The computed orders of $\|\bmr{p} -
\bmr{p}_h\|_{L^2(\Omega)}$, $\enorm{u - u_h}_u$ and $\|u -
u_h\|_{L^2(\Omega)}$ are about $1$. A possible explanation of the
rates may be traced back to the lack of $H^3$-regularity of the exact
solution on the whole domain.

\begin{figure}[!htp]
  \centering
  \includegraphics[width=0.4\textwidth]{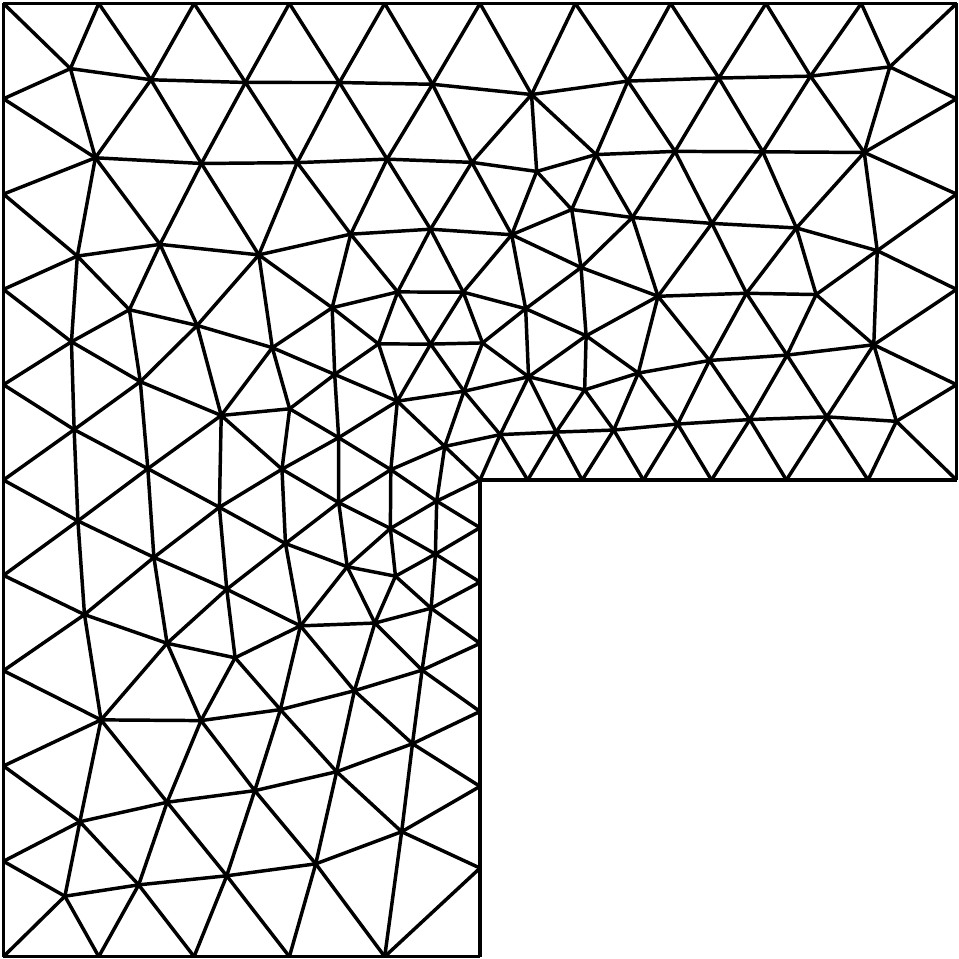}
  \hspace{25pt}
  \includegraphics[width=0.4\textwidth]{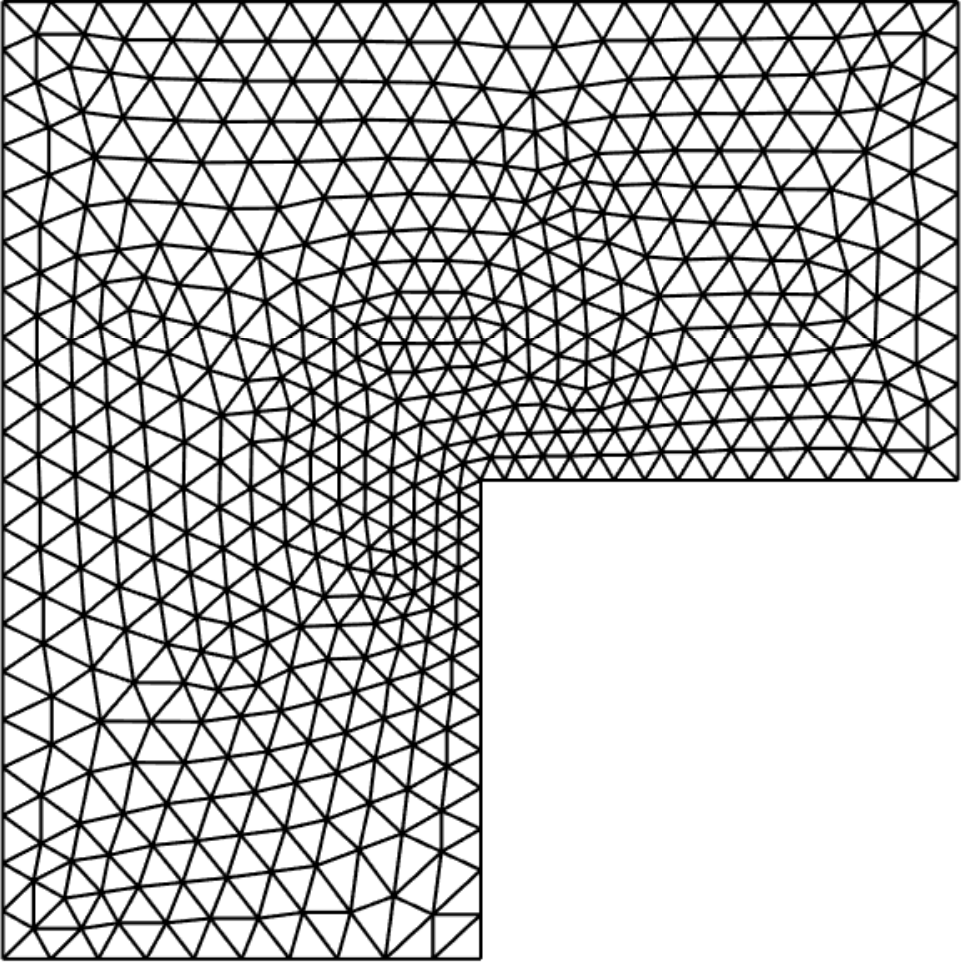}
  \caption{The triangular meshes with 250 elements (left) and
  1000 elements for Example 3.}
  \label{fig:Lshapetriangulation}
\end{figure}

\begin{table}
  \centering
  \renewcommand\arraystretch{1.2}
  \begin{tabular}{p{0.3cm}| p{0.2cm} p{1.8cm} p{0.8cm} p{1.8cm}
    p{0.8cm} p{1.8cm} p{0.8cm} p{1.8cm} p{0.8cm}}
    \hline\hline
    $m$ & & $\|e_u\|_{L^2(\Omega)} $ & order & $\enorm{e_u}_u$ & order
    & $\| e_{\bmr{p}} \|_{\bmr{L}^2(\Omega)} $ & order &
    $\enorm{e_{\bmr{p}}}_{\bmr{p}}$ & order \\
    \hline
    \multirow{5}{*}{$1$}
   & 1 & 4.5059e-03 & -    & 1.5490e-01 & -    
       & 3.9382e-02 & -    & 4.2524e-02 & -    \\
   & 2 & 1.3528e-03 & 1.73 & 7.7746e-02 & 0.99 
       & 1.8681e-02 & 1.07 & 2.5539e-02 & 0.73 \\
   & 3 & 4.0795e-04 & 1.73 & 3.8900e-02 & 1.00 
       & 9.3165e-03 & 1.00 & 1.5859e-02 & 0.68 \\
   & 4 & 1.3376e-04 & 1.61 & 1.9563e-02 & 1.00 
       & 4.5537e-03 & 1.03 & 9.9689e-03 & 0.67 \\
   & 5 & 5.2105e-05 & 1.36 & 9.7263e-03 & 1.00 
       & 2.2927e-03 & 1.00 & 6.2815e-03 & 0.67 \\
    \hline
    \multirow{5}{*}{$2$}
   & 1 & 2.2627e-03 & -    & 9.3186e-03 & -    
       & 3.4672e-02 & -    & 5.1619e-02 & - \\
   & 2 & 6.8183e-04 & 1.73 & 2.6548e-03 & 1.81 
       & 1.6061e-02 & 1.11 & 2.9373e-02 & 0.81 \\
   & 3 & 2.3956e-04 & 1.51 & 9.2329e-04 & 1.52 
       & 8.0869e-03 & 0.99 & 1.8481e-02 & 0.67 \\
   & 4 & 1.0011e-04 & 1.99 & 3.7505e-04 & 1.29 
       & 4.0509e-03 & 1.26 & 1.1383e-02 & 0.68 \\
   & 5 & 4.5381e-05 & 1.13 & 1.7137e-04 & 1.12 
       & 2.0293e-03 & 1.00 & 7.0855e-03 & 0.68 \\
   \hline
   \multirow{5}{*}{$3$}
   & 1 & 2.5557e-03 & -    & 1.1823e-02 & -    
       & 4.1292e-02 & -    & 5.7175e-02 & - \\
   & 2 & 8.6799e-04 & 1.55 & 4.1778e-03 & 1.50 
       & 1.9767e-02 & 1.06 & 3.0635e-02 & 0.90 \\
   & 3 & 3.3653e-04 & 1.36 & 1.4712e-03 & 1.50 
       & 9.6459e-03 & 1.03 & 1.0801e-02 & 0.76 \\
   & 4 & 1.5550e-04 & 1.13 & 5.9787e-04 & 1.29 
       & 4.9361e-05 & 0.98 & 1.1136e-02 & 0.68 \\
   & 5 & 7.5031e-05 & 1.06 & 2.8188e-04 & 1.08 
       & 2.5011e-05 & 0.99 & 6.9361e-03 & 0.68 \\
   \hline\hline
  \end{tabular}
  \caption{Example 4. The errors $e_u = u - u_h, e_{\bmr{p}} = \bmr{p}
  - \bmr{p}_h$, and the orders of convergence with the spaces
  $\bmr{U}_h^{m} \times \wh{V}_h^m(1 \leq m \leq 3)$.}
  \label{tab:ex4error}
\end{table}

\paragraph{\bf Example 5} We consider a three-dimensional Poisson
problem on a unit cube $\Omega = [0, 1]^3$. The domain is partitioned
into a series of tetrahedral meshes with mesh size $h = 1/5, 1/10,
1/20, 1/40$ by {\tt Gmsh} \cite{geuzaine2009gmsh}. The exact solution
is taken as 
\begin{displaymath}
  u(x, y, z) = \sin(2\pi x) \sin(2\pi y) \sin(2 \pi z),
\end{displaymath}
and the Dirichlet function $g$ and the source term $f$ are taken
suitably. We use the spaces $\bmr{U}_h^m \times \wh{V}_h^m(1 \leq m
\leq 3)$ to approximate $\bmr{p}$ and $u$, respectively. The numerical
results are presented in Tab.~\ref{tab:ex5errorm}. We still observe
the optimal convergence rate for $\bmr{p}_h$ under $\bmr{L}^2$ norm
when $m$ is odd, and all computed convergence orders agree with the
theoretical analysis. 

\begin{table}
  \centering
  \renewcommand\arraystretch{1.2}
  \begin{tabular}{p{0.3cm}| p{0.2cm} p{1.8cm} p{0.8cm} p{1.8cm}
    p{0.8cm} p{1.8cm} p{0.8cm} p{1.8cm} p{0.8cm}}
    \hline\hline
    $m$ & & $\|e_u\|_{L^2(\Omega)} $ & order & $\enorm{e_u}_u$ & order
    & $\| e_{\bmr{p}} \|_{\bmr{L}^2(\Omega)} $ & order &
    $\enorm{e_{\bmr{p}}}_{\bmr{p}}$ & order \\
    \hline
   \multirow{5}{*}{$1$} 
   & 1 & 2.0159e-01 & -    & 2.6227e-00 & -    
       & 1.4772e-00 & -    & 2.0737e+01 & -    \\
   & 2 & 6.7739e-02 & 1.76 & 1.4117e-00 & 0.89 
       & 4.3453e-01 & 1.80 & 1.0927e+01 & 0.93 \\
   & 3 & 1.8200e-02 & 1.90 & 7.3125e-01 & 0.95 
       & 1.1641e-02 & 1.90 & 5.4683e+00 & 0.99 \\
   & 4 & 4.6456e-03 & 1.96 & 3.6691e-01 & 1.00 
       & 2.9923e-02 & 1.96 & 2.7331e+00 & 1.00 \\
    \hline
    \multirow{5}{*}{$2$}
   & 1 & 2.8293e-02 & -    & 7.6111e-01 & -    
       & 3.6002e-01 & -    & 7.0288e+00 & - \\
   & 2 & 9.1341e-02 & 1.63 & 2.2963e-01 & 1.73 
       & 1.0421e-01 & 1.79 & 1.7895e+00 & 1.97 \\
   & 3 & 2.5926e-03 & 1.82 & 6.1281e-02 & 1.91 
       & 2.8129e-02 & 1.89 & 4.6372e-01 & 1.95 \\
   & 4 & 6.8012e-04 & 1.93 & 1.5021e-02 & 2.01 
       & 7.2823e-03 & 1.95 & 1.1599e-01 & 2.00 \\
   \hline
   \multirow{5}{*}{$3$}
   & 1 & 7.2877e-03 & -    & 1.8326e-01 & -    
       & 1.7658e-01 & -    & 3.0434e+00 & - \\
   & 2 & 7.3997e-04 & 3.30 & 2.1873e-02 & 3.06 
       & 1.3510e-02 & 3.71 & 3.9250e-01 & 2.96 \\
   & 3 & 5.6061e-05 & 3.73 & 2.7168e-03 & 3.28 
       & 9.2336e-04 & 3.87 & 5.1203e-02 & 3.01 \\
   & 4 & 3.6203e-06 & 3.96 & 3.3962e-04 & 3.17 
       & 5.9170e-05 & 3.96 & 6.4123e-03 & 3.00 \\
   \hline\hline
  \end{tabular}
  \caption{Example 5. The errors $e_u = u - u_h, e_{\bmr{p}} = \bmr{p}
  - \bmr{p}_h$, and the orders of convergence with the spaces
  $\bmr{U}_h^m \times \wh{V}_h^m(1 \leq m \leq 3)$.}
  \label{tab:ex5errorm}
\end{table}

\subsection{Efficiency comparison}
The number of the degrees of freedom of a discretized system is a
suitable indicator for the efficiency, as illustrated by Hughes et al
in \cite{hughes2000comparison}. In our method, the accuracy of
$\bmr{p}_h$ determines the convergence behavior of the pressure. Thus,
to show the efficiency of the proposed method, we make a comparison
between the standard least squares discontinuous finite element method
presented in Section~\ref{sec:fems} and the proposed method by
comparing the error of the numerical flux $\bmr{p}_h$. 

For both methods, we select the finite element spaces of equal order
for solving the Poisson problem. Here we solve the problems that are
taken from the Example 1 and Example 5 for two and three dimensional
case, respectively. We implement the two methods on successively
refined meshes. In Fig.~\ref{fig:compared2}, we plot the errors of
numerical flux in the DLS energy norm $\| \cdot \|_{\bmr{p}}$ against
the number of degrees of freedom with $1 \leq m \leq 3$ in two and
three dimension. All convergence orders are in perfect agreement
with the theoretical results.

There are two points notable for us. To achieve the same accuracy, the
proposed method uses much less DOFs than the DLS finite element
method. The saving of number of DOFs is more remarkable for higher
order approximation. For $d=2$, the number of DOFs used in our method
is about $36\%$ of that in DLS method for linear approximation to
achieve the same accuracy. Meanwhile, the number of DOFs used in our
method is about $31\%$ and $27\%$ of the number of DOFs used in DLS
method for $m=2$ and $3$, respectively (see
Fig.~\ref{fig:compared2}). In Fig.~\ref{fig:compared2}, one may see
that the saving of number DOFs for 3D problems is even more
significant than 2D problems. For $d=3$, the percentages of number of
DOFs reduce to about $30\%$, $12\%$, and $5\%$ of that in DLS method
for $m=1$, $2$, and $3$, respectively.

Let us note at last that the numerical flux $\bmr{p}_h$ obtained by
our method is locally irrotational, which is a natural property as the
gradient of a function.

\begin{figure}[!htp]
  \centering
  \includegraphics[width=0.45\textwidth]{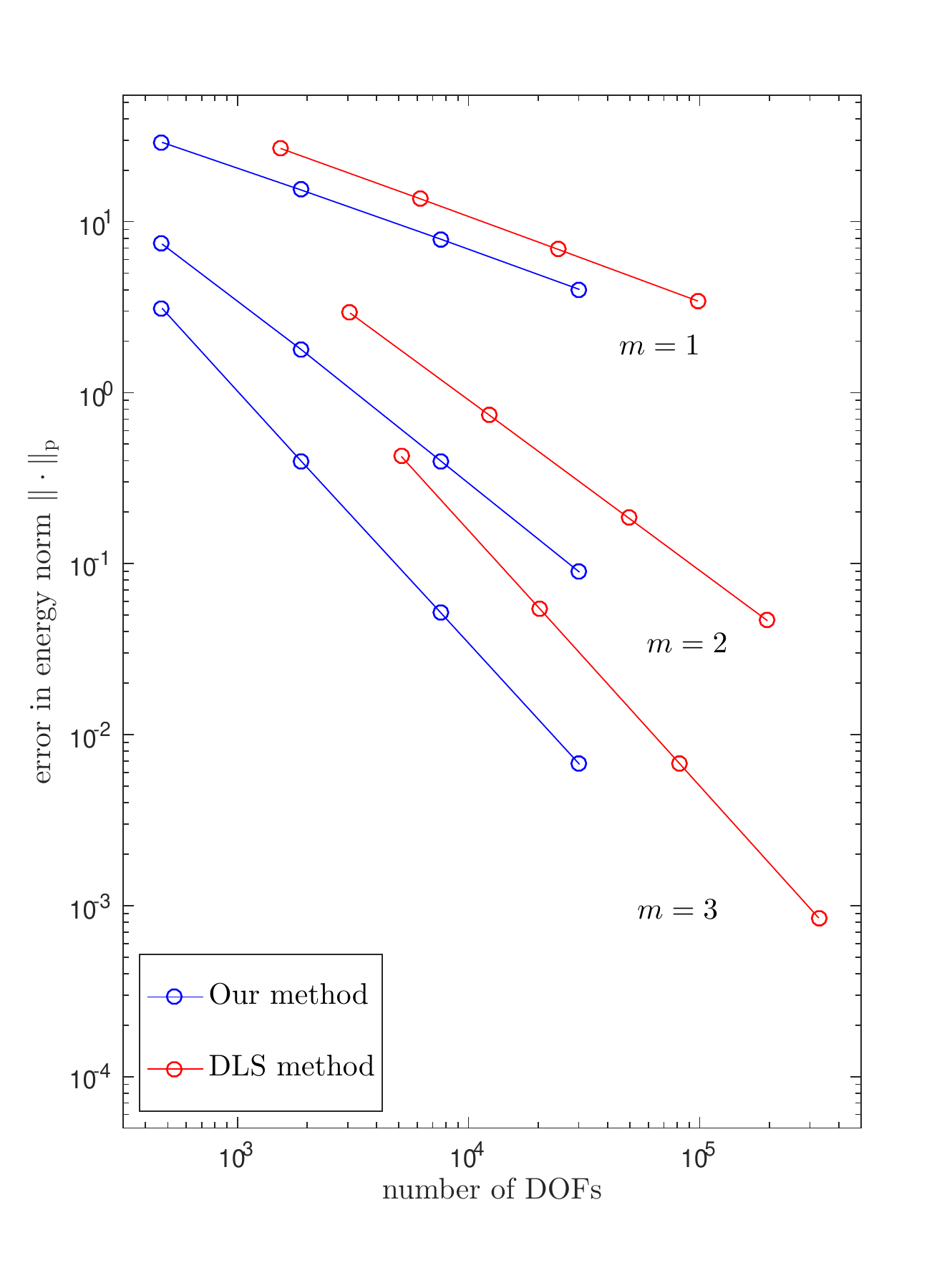}
  \hspace{25pt}
  \includegraphics[width=0.45\textwidth]{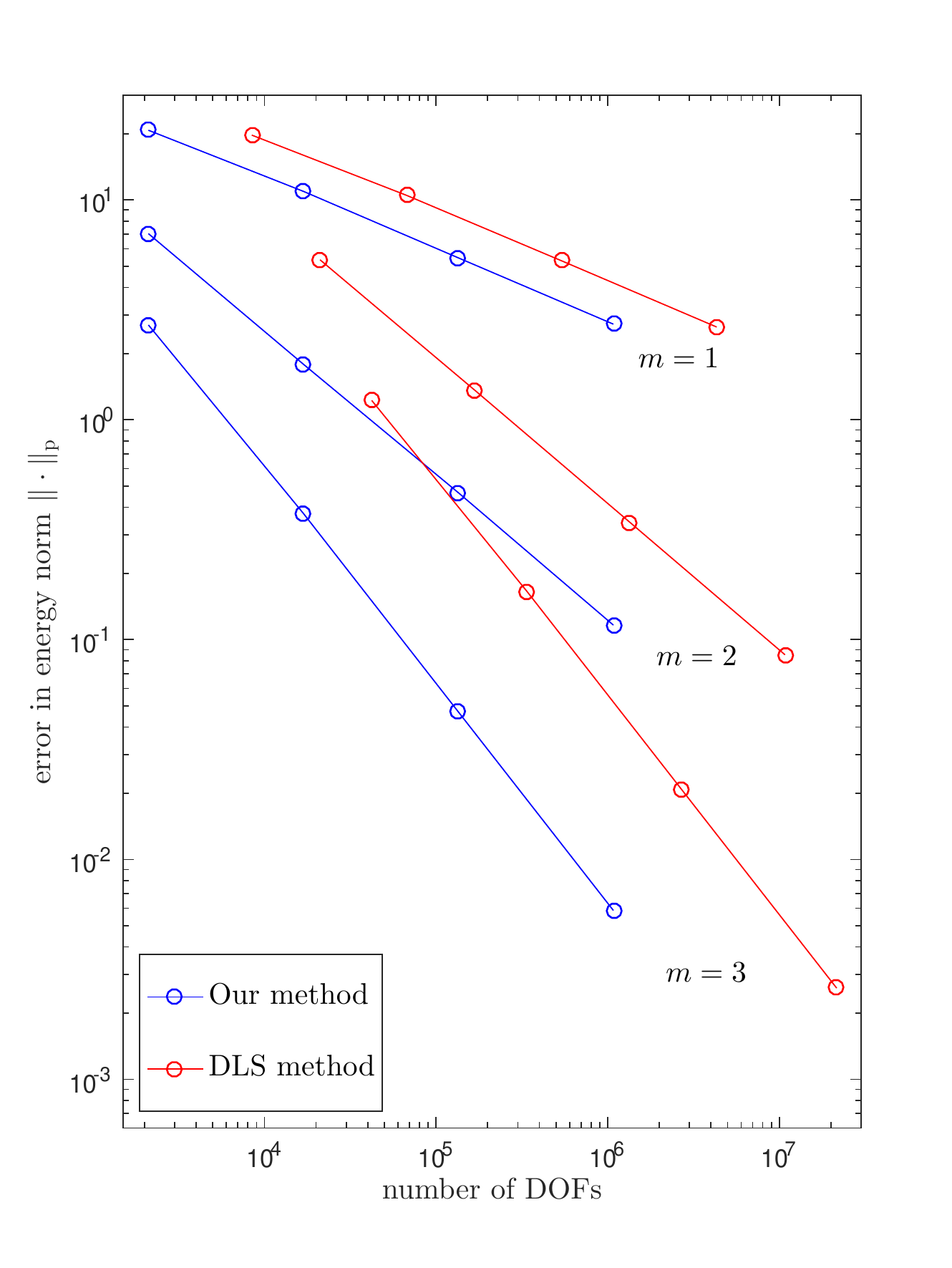}
  \caption{Comparison of the error $\|\bmr{p} - \bmr{p}_h\|_{\bmr{p}}$
  in number of DOFs by two methods with $m =1, 2, 3$ in two
  dimension (left) and three dimension (right).}
  \label{fig:compared2}
\end{figure}


\section{Conclusion}
We proposed a sequential least squares finite element method for the
Poisson equation. The novel piecewisely irrotational approximation
space is constructed by solving local least squares problem and we use
this space to decouple the least squares minimization problem. We
proved the convergences for pressure and flux in $L^2$ norm and energy
norm. By a series of numerical results, not only the error estimates
are verified, but also we exhibited the flexibility and the great
efficiency of our method.

\section*{Acknowledgements}
This research is supported by the National Natural Science Foundation
of China (Grant No. 91630310, 11421110001, and 11421101) and the
Science Challenge Project, No. TZ2016002.


\begin{appendix}
  \section{}
  In Appendix, we present some details of the reconstruction process.
  We first give an example of constructing the element patch in two
  dimensional case. For element $K$, the construction of $S(K)$ with
  $\# S(K) = 15$ is presented in Fig.~\ref{fig:buildpatch}. 
  \begin{figure}
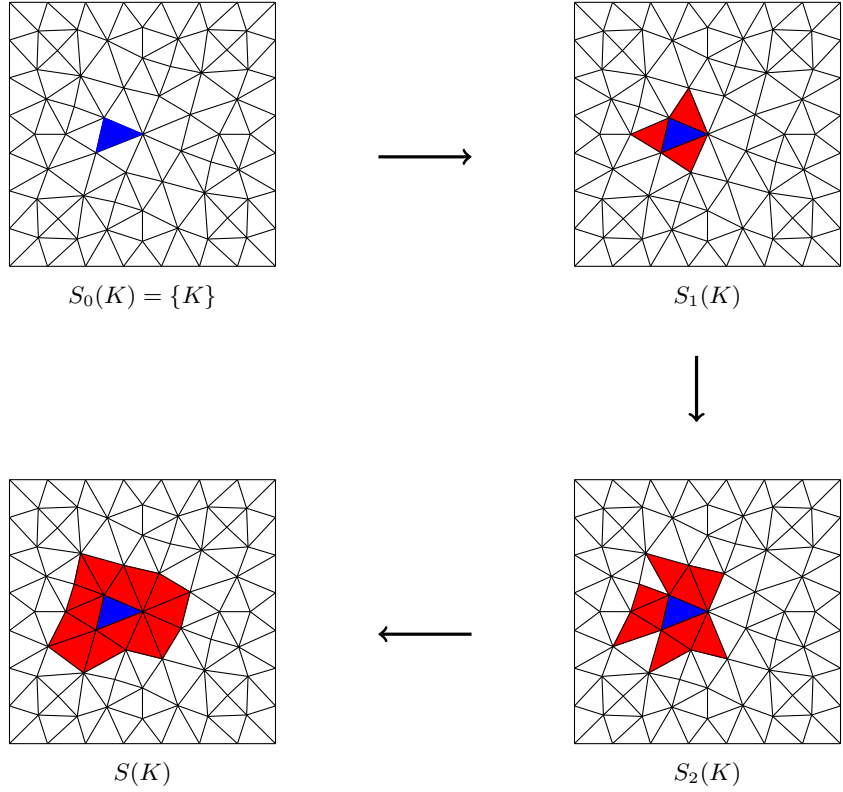

    \centering \captionsetup[subfigure]{labelformat=empty}
    \begin{subfigure}{.35\textwidth}
      \centering
      \begin{tikzpicture}[scale=3.5]
        \input{./figure/m0.tex}
      \end{tikzpicture}
      \caption{$S_0(K)=\{K\}$}
    \end{subfigure}
    \begin{subfigure}{.1\textwidth}
      \centering
      \begin{tikzpicture}[scale=3.5]
        \draw[very thick, ->] (0.35, 0.5) -- (0.7, 0.5);
      \end{tikzpicture}
    \end{subfigure}
    \begin{subfigure}{.35\textwidth}
      \centering
      \begin{tikzpicture}[scale=3.5]
        \input{./figure/m1.tex}
      \end{tikzpicture}
      \caption{$S_1(K)$}
    \end{subfigure}

    \vspace{15pt}
    \hspace{200pt}
    \begin{subfigure}{.1\textwidth}
        \centering
      \begin{tikzpicture}[scale=3.5]
        \draw[very thick, <-, rotate=90] (0.5, 0.5) -- (0.75, 0.5);
      \end{tikzpicture}
    \end{subfigure}

    \vspace{20pt}
    \begin{subfigure}{.35\textwidth}
      \centering
      \begin{tikzpicture}[scale=3.5]
        \input{./figure/m3.tex}
      \end{tikzpicture}
      \caption{$S(K)$}
    \end{subfigure}
    \begin{subfigure}{.1\textwidth}
      \centering
      \begin{tikzpicture}[scale=3.5]
        \draw[very thick, <-] (0.35, 0.5) -- (0.7, 0.5);
      \end{tikzpicture}
    \end{subfigure}
    \begin{subfigure}{.35\textwidth}
      \centering
      \begin{tikzpicture}[scale=3.5]
        \input{./figure/m2.tex}
      \end{tikzpicture}
      \caption{$S_2(K)$}
    \end{subfigure}
    \caption{Build patch for element $K$ with $\# S(K) = 15$}
    \label{fig:buildpatch}
  \end{figure}
  Then we give more details about the space $\bmr{U}_h^m$.  As we
  mentioned before, the operator $\mc R^m$ embeds the space
  $\bmr{C}^0(\Omega) \cap H(\curl^0; \Omega)$ to the piecewise
  irrotational polynomial space of degree $m$ by solving the local
  least squares problem. We define $\bmr{w}_K^i(\bm x) \in
  \bmr{C}^0(\Omega)(1 \leq i \leq d)$ that
  \begin{displaymath}
    \bmr{w}_K^i(\bm x) = \begin{cases}
      \bm e_i, \quad \bm x = \bm x_K, \\
      \bm 0, \quad \bm x \in \widetilde{K}, \quad \widetilde{K} \neq
      K,
    \end{cases} 
    \quad \forall K \in \MTh,
  \end{displaymath}
  where $\bm e_i$ is a $d \times 1$ unit vector whose $i$-th entry is
  $1$.  Then $\bmr{U}_{h}^m = \text{span}\{\bmr{\lambda}_K^i\ |\
  \bmr{\lambda}_K^i = \mc R^m \bmr{w}_K^i,\ 1 \leq i \leq d, \ K \in
  \MTh\}$, and one can write the operator $\mc R^m$ in an explicit
  way: for a function $\bmr{g} = (g^1, \cdots, g^d) \in
  \bmr{C}^0(\Omega) \cap H(\curl^0; \Omega)$ we have
  \begin{displaymath}
    \mc R^m \bmr{g} = \sum_{K \in \MTh} \sum_{i=1}^d g^i(\bm x_{K})
    \bmr{\lambda}_K^i(\bm x).
  \end{displaymath}
  Clearly, the number of DOFs of our method is always $d$ times the
  number of elements in partition.

  Further, we give some details about the computer implementation of
  the reconstructed space. We take the case $d = 2$ to illustrate. We
  first outline the bases of the space $\bmr{S}_m(D)$, it is easily
  verified that for $d=2$, 
  \begin{displaymath}
    \bmr{S}_1(D) = \left\{ \vecd{1}{0}, \vecd{0}{1}, \vecd{x}{0},
    \vecd{0}{y}, \vecd{y}{x} \right\}.
  \end{displaymath}
  Similarly for $m=2, 3$, there is 
  \begin{displaymath}
    \begin{aligned}
      \bmr{S}_2(D) = \bmr{S}_1(D) \cup &\left\{ \vecd{x^2}{0},
      \vecd{2xy}{x^2}, \vecd{y^2}{2xy}, \vecd{0}{y^2} \right\}, \\
      \bmr{S}_3(D) = \bmr{S}_2(D) \cup &\left\{ \vecd{x^3}{0},
      \vecd{3x^2}{y}, \vecd{2xy^2}{2x^2y}, \vecd{y^3}{3xy^2},
      \vecd{0}{y^3} \right\}.
    \end{aligned}
  \end{displaymath}
  Then we shall solve the least squares problem \eqref{eq:lsproblem}
  on every element. We take $K_0$ and $m = 1$ for an instance (see
  Fig.~\ref{fig:Kexample}), and we let $S(K_0) = \left\{ K_0, K_1,
  K_2, K_3 \right\}$ where $K_i(i=1,2,3)$ are the adjacent
  edge-neighbouring elements of $K_0$. We denote by $\bm{x}_i = (x_i,
  y_i)$ the barycenter of the element $K_i$ and we obtain the
  collocation points set $\mc I_{K_0} = \left\{ \bm{x}_0, \bm{x}_1,
  \bm{x}_2, \bm{x}_3 \right\}$. 
  \begin{figure}[htp]
    \centering
    \begin{tikzpicture}[scale=1]
      \coordinate (A) at (1, 0); 
      \coordinate (B) at (-0.5, -0.6);
      \coordinate (C) at (-0.5, 0.8);
      \coordinate (D) at (1.2, 1.5);
      \coordinate (E) at (-2, 0);
      \coordinate (F) at (0.8, -1.5);
      \draw[fill, red] (A) -- (B) -- (C);
      \draw[thick, black] (A) -- (C) -- (B) -- (A);
      \draw[thick, black] (A) -- (D) -- (C);
      \draw[thick, black] (C) -- (E) -- (B);
      \draw[thick, black] (A) -- (F) -- (B);
      \node at(0, 0) {$K_0$}; \node at(1.7/3, 2.3/3) {$K_1$};
      \node at(1.3/3, -2.1/3) {$K_2$}; \node at(-1, 0.2/3) {$K_3$};
    \end{tikzpicture}
    \caption{$K$ and its neighbours}
    \label{fig:Kexample}
  \end{figure}
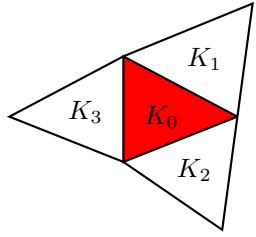
  Then for the function $\bmr{g} = (g^1, g^2) \in \bmr{C}^0(\Omega)
  \cap H(\curl^0; \Omega)$ the least squares problem on $K_0$ reads
  \begin{displaymath}
    \mathop{\arg\min}_{\bm a \in \mb R^5} \sum_{i = 0}^3 \left\| a_0
    \vecd{1}{0} + a_1 \vecd{0}{1} + a_2 \vecd{x_i}{0} + a_3
    \vecd{0}{y_i} + a_4\vecd{y_i}{x_i} - \vecd{g^1(x_i)}{g^2(y_i)}
    \right\|^2.   
  \end{displaymath}
  It is easy to obtain its unique solution 
  \begin{displaymath}
    \bm a = (A^TA)^{-1} A^T\bm q,
  \end{displaymath}
  where 
  \begin{displaymath}
    A = \begin{bmatrix}
      1 & 0 & x_0 & 0   & y_0 \\
      0 & 1 & 0   & y_0 & x_0 \\
      1 & 0 & x_1 & 0   & y_1 \\
      0 & 1 & 0   & y_1 & x_1 \\
      1 & 0 & x_2 & 0   & y_2 \\
      0 & 1 & 0   & y_2 & x_2 \\
      1 & 0 & x_3 & 0   & y_3 \\
      0 & 1 & 0   & y_3 & x_3 \\
    \end{bmatrix}, \qquad 
    \bm{q} = \begin{bmatrix}
      g^1(x_0) \\ g^2(y_0) \\  g^1(x_1) \\ g^2(y_1) \\ 
      g^1(x_2) \\ g^2(y_2) \\  g^1(x_3) \\ g^2(y_3) \\
    \end{bmatrix}.
  \end{displaymath}
  We notice that the matrix $(A^TA)^{-1}A^T$ is independent of the
  function $\bmr g$ and includes all information of the function
  $\bmr{\lambda}_{K_j}^i(j=0,1,2,3,\ i=1,2)$ on the element $K_0$.
  Thus we could store the matrix $(A^TA)^{-1}A^T$ for every element to
  represent our approximation space. The idea of the implementation
  could be adapted to the high-order accuracy case and the high
  dimensional problem without any difficulty.
\end{appendix}

\bibliographystyle{amsplain}
\bibliography{../../ref}

\providecommand{\bysame}{\leavevmode\hbox to3em{\hrulefill}\thinspace}
\providecommand{\MR}{\relax\ifhmode\unskip\space\fi MR }
\providecommand{\MRhref}[2]{%
  \href{http://www.ams.org/mathscinet-getitem?mr=#1}{#2}
}
\providecommand{\href}[2]{#2}
\begin{thebibliography}{10}

\bibitem{antonietti:2013}
P.~F. Antonietti, L.~Beir\~ao~da Veiga, and M.~Verani, \emph{A mimetic
  discretization of elliptic obstacle problems}, Math. Comp. \textbf{82}
  (2013), no.~283, 1379--1400.

\bibitem{arnold1982interior}
D.~N. Arnold, \emph{An interior penalty finite element method with
  discontinuous elements}, SIAM J. Numer. Anal. \textbf{19} (1982), no.~4,
  742--760.

\bibitem{arnold2002unified}
D.~N. Arnold, F.~Brezzi, B.~Cockburn, and L.~D. Marini, \emph{Unified analysis
  of discontinuous {G}alerkin methods for elliptic problems}, SIAM J. Numer.
  Anal. \textbf{39} (2001/02), no.~5, 1749--1779.

\bibitem{Aziz1985least}
A.~K. Aziz, R.~B. Kellogg, and A.~B. Stephens, \emph{Least squares methods for
  elliptic systems}, Math. Comp. \textbf{44} (1985), no.~169, 53--70.
  \MR{771030}

\bibitem{Bensow2005div}
Rickard Bensow and Mats~G. Larson, \emph{Discontinuous least-squares finite
  element method for the div-curl problem}, Numer. Math. \textbf{101} (2005),
  no.~4, 601--617. \MR{2195400}

\bibitem{Bensow2005discontinuous}
Rickard~E. Bensow and Mats~G. Larson, \emph{Discontinuous/continuous
  least-squares finite element methods for elliptic problems}, Math. Models
  Methods Appl. Sci. \textbf{15} (2005), no.~6, 825--842.

\bibitem{Bochev2012locally}
Pavel Bochev, James Lai, and Luke Olson, \emph{A locally conservative,
  discontinuous least-squares finite element method for the {S}tokes
  equations}, Internat. J. Numer. Methods Fluids \textbf{68} (2012), no.~6,
  782--804. \MR{2878612}

\bibitem{Bochev2013nonconforming}
\bysame, \emph{A non-conforming least-squares finite element method for
  incompressible fluid flow problems}, Internat. J. Numer. Methods Fluids
  \textbf{72} (2013), no.~3, 375--402. \MR{3049438}

\bibitem{Bochev1993accuracy}
Pavel~B. Bochev and Max~D. Gunzburger, \emph{Accuracy of least-squares methods
  for the {N}avier-{S}tokes equations}, Comput. \& Fluids \textbf{22} (1993),
  no.~4-5, 549--563.

\bibitem{Bochev1998review}
\bysame, \emph{Finite element methods of least-squares type}, SIAM Rev.
  \textbf{40} (1998), no.~4, 789--837. \MR{1659689}

\bibitem{Bochev2009least}
\bysame, \emph{Least-squares finite element methods}, Applied Mathematical
  Sciences, vol. 166, Springer, New York, 2009.

\bibitem{Bramble1997least}
James~H. Bramble, Raytcho~D. Lazarov, and Joseph~E. Pasciak, \emph{A
  least-squares approach based on a discrete minus one inner product for first
  order systems}, Math. Comp. \textbf{66} (1997), no.~219, 935--955.
  \MR{1415797}

\bibitem{Lung1994stokes}
Ching~Lung Chang, \emph{An error estimate of the least squares finite element
  method for the {S}tokes problem in three dimensions}, Math. Comp. \textbf{63}
  (1994), no.~207, 41--50. \MR{1234425}

\bibitem{ciarlet2002finite}
P.~G. Ciarlet, \emph{{The Finite Element Method for Elliptic Problems}},
  Classics in Applied Mathematics, vol.~40, Society for Industrial and Applied
  Mathematics (SIAM), Philadelphia, PA, 2002, Reprint of the 1978 original
  [North-Holland, Amsterdam; MR0520174 (58 \#25001)].

\bibitem{geuzaine2009gmsh}
C.~Geuzaine and J.~F. Remacle, \emph{Gmsh: {A} 3-{D} finite element mesh
  generator with built-in pre- and post-processing facilities}, Internat. J.
  Numer. Methods Engrg. \textbf{79} (2009), no.~11, 1309--1331.

\bibitem{girault1986finite}
Vivette Girault and Pierre~Arnaud Raviart, \emph{Finite element methods for
  navier-stokes equations: Theory and algorithms}, Springer-Verlag, 1986.

\bibitem{hughes2000comparison}
Thomas J.~R. Hughes, Gerald Engel, Luca Mazzei, and Mats~G. Larson, \emph{A
  comparison of discontinuous and continuous {G}alerkin methods based on error
  estimates, conservation, robustness and efficiency}, Discontinuous {G}alerkin
  methods ({N}ewport, {RI}, 1999), Lect. Notes Comput. Sci. Eng., vol.~11,
  Springer, Berlin, 2000, pp.~135--146. \MR{1842169}

\bibitem{Jiang1993optimal}
Bo-Nan Jiang and Louis~A. Povinelli, \emph{Optimal least-squares finite element
  method for elliptic problems}, Comput. Methods Appl. Mech. Engrg.
  \textbf{102} (1993), no.~2, 199--212.

\bibitem{li2017discontinuous}
R.~Li, P.~B. Ming, Z.~Y. Sun, F.~Y. Yang, and Z.~J. Yang, \emph{A discontinuous
  {G}alerkin method by patch reconstruction for biharmonic problem}, accepted
  by Journal of Computational Mathematics, arXiv:1712.10103 (2017).

\bibitem{li2016discontinuous}
R.~Li, P.~B. Ming, Z.~Y. Sun, and Z.~J. Yang, \emph{An arbitrary-order
  discontinuous {G}alerkin method with one unknown per element},
  arXiv:1803.00378 (2018).

\bibitem{li2012efficient}
R.~Li, P.~B. Ming, and F.~Tang, \emph{An efficient high order heterogeneous
  multiscale method for elliptic problems}, Multiscale Model. Simul.
  \textbf{10} (2012), no.~1, 259--283.

\bibitem{Mitchell2013collection}
William~F. Mitchell, \emph{A collection of 2{D} elliptic problems for testing
  adaptive grid refinement algorithms}, Appl. Math. Comput. \textbf{220}
  (2013), 350--364.

\bibitem{Mitchell2015high}
William~F. Mitchell, \emph{How high a degree is high enough for high order
  finite elements?}, Procedia Computer Science \textbf{51} (2015), 246 -- 255,
  International Conference On Computational Science, ICCS 2015.

\bibitem{Pehlivanov1994least}
A.~I. Pehlivanov, G.~F. Carey, and R.~D. Lazarov, \emph{Least-squares mixed
  finite elements for second-order elliptic problems}, SIAM J. Numer. Anal.
  \textbf{31} (1994), no.~5, 1368--1377. \MR{1293520}

\bibitem{talischi2012polymesher}
C.~Talischi, G.~H. Paulino, A.~Pereira, and I.~F.~M. Menezes, \emph{{\tt
  {P}oly{M}esher}: a general-purpose mesh generator for polygonal elements
  written in {M}atlab}, Struct. Multidiscip. Optim. \textbf{45} (2012), no.~3,
  309--328.

\bibitem{Ye2018discontinuous}
Xiu Ye and Shangyou Zhang, \emph{A discontinuous least-squares finite-element
  method for second-order elliptic equations}, International Journal of
  Computer Mathematics \textbf{96} (2019), no.~3, 557--567.

\bibitem{zienkiewicz2003discontinuous}
O.~C. Zienkiewicz, R.~L. Taylor, S.~J. Sherwin, and J.~Peir\'o, \emph{On
  discontinuous {G}alerkin methods}, Internat. J. Numer. Methods Engrg.
  \textbf{58} (2003), no.~8, 1119--1148.

\end{thebibliography}

\end{document}